 \documentclass[reqno,11pt]{amsart}
 \usepackage{amsmath, amsthm, a4, latexsym, amssymb}
\usepackage[unicode]{hyperref}
\usepackage{srcltx}
\usepackage{xcolor}
\usepackage{mathrsfs} 
\usepackage{graphicx}

\usepackage{mathrsfs}

\def\@rmrk#1#2{\refstepcounter
    {#1}\@ifnextchar[{\@yrmrk{#1}{#2}}{\@xrmrk{#1}{#2}}}

\makeatletter\@addtoreset{equation}{section}\makeatother

 \newfont{\bfit}{cmbxti10 scaled 1200}

\renewcommand{\d}{{\rm d}}

 \newcommand{\e}{{\rm e} }

 \newcommand{\eps}{\varepsilon}

 \newcommand{\R}{\mathbb{R}}
 \newcommand{\N}{\mathbb{N}}

 \newcommand{\E}{\mathbf{E}}
 \newcommand{\bE}{\mathbb{E}}
 \renewcommand{\P}{\mathbf{P}}
 \newcommand{\bP}{\mathbb{P}}
 \newcommand{\bQ}{\mathbb{Q}}
 \newcommand{\hP}{\widehat{\mathscr {M}}}
 \def\1{{\mathchoice {1\mskip-4mu\mathrm l} 
{1\mskip-4mu\mathrm l}
{1\mskip-4.5mu\mathrm l} {1\mskip-5mu\mathrm l}}}

 \newcommand{\Mcal}{{\mathcal M}}

  \newcommand{\Bt}{{\boldsymbol{{\dot B}}_t}}
  \newcommand{\Br}{{\boldsymbol{{\dot B}}_r}}
   \newcommand{\Bs}{{\boldsymbol{{\dot B}}_s}}
    
 \newcommand{\bET}{\widehat{\mathbb E}_T}
  \newcommand{\bEt}{\widehat{\mathbb E}_t}
\newcommand{\bETt}{\widehat{\mathbb E}_T^{\ssup{\boldsymbol{{\dot B}}_t}}}

\newcommand{\X}{\widetilde{\mathcal X}}

\newcommand{\ignore}[1]{}

\newcommand{\ssup}[1] {{\scriptscriptstyle{({#1}})}}

\renewcommand{\subsection}{\secdef \subsct\sbsect}
\newcommand{\subsct}[2][default]{\refstepcounter{subsection}
\vspace{0.15cm}
{\flushleft\bf \arabic{section}.\arabic{subsection}~\bf #1  }
\nopagebreak\nopagebreak}
\newcommand{\sbsect}[1]{\vspace{0.1cm}\noindent
{\bf #1}\vspace{0.1cm}}

{\nopagebreak {\hfill\rule{2mm}{2mm}}\\ }

\newtheorem{theorem}{Theorem}[section]
\newtheorem{lemma}[theorem]{Lemma}
\newtheorem{cor}[theorem]{Corollary}
\newtheorem{prop}[theorem]{Proposition}

\newtheoremstyle{thm}{1.5ex}{1.5ex}{\itshape\rmfamily}{}
{\bfseries\rmfamily}{}{2ex}{}

\newtheoremstyle{rem}{1.3ex}{1.3ex}{\rmfamily}{}
{\itshape\rmfamily}{}{1.5ex}{}
\theoremstyle{rem}
\newtheorem{remark}{{\slshape\sffamily Remark}}[]

\refstepcounter{subsubsection}

\def\thebibliography#1{\section*{References}
  \list%
  {\arabic{enumi}.}
    {\settowidth\labelwidth{[#1]}\leftmargin\labelwidth
    \advance\leftmargin\labelsep
    \parsep0pt\itemsep0pt
    \usecounter{enumi}}
    \def\newblock{\hskip .11em plus .33em minus .07em}
    \sloppy                   
    \sfcode`\.=1000\relax}

\begin{document}

\title[Geometry of the GMC in the Wiener space]
{\large Geometry of the Gaussian multiplicative chaos in the Wiener space}
\author[Yannic Br\"oker and Chiranjib Mukherjee]{}
\maketitle
\thispagestyle{empty}
\vspace{-0.5cm}

\centerline{\sc Yannic Br\"oker  \footnote{University of Muenster, Einsteinstrasse 62, Muenster 48149, Germany, {\tt yannic.broeker@uni-muenster.de}} and 
Chiranjib Mukherjee \footnote{University of  Muenster, Einsteinstrasse 62, Muenster 48149, Germany, {\tt chiranjib.mukherjee@uni-muenster.de}
}}

\renewcommand{\thefootnote}{}

\vspace{-0.5cm}
\centerline{\textit{University of M\"unster}}
\vspace{0.2cm}

\begin{center}
\today

\end{center}

\begin{quote}{\small {\bf Abstract: } 
We develop an approach for investigating geometric properties of Gaussian multiplicative chaos (GMC) in an infinite dimensional set up. The base space is chosen to be the space of continuous functions endowed with Wiener measure, and the random field is a space-time white noise integrated against Brownian paths. In this set up, we show that in any dimension $d\geq 1$ and for any inverse temperature, the GMC-volume of a ball, uniformly around all paths, decays exponentially with an explicit decay rate. The exponential rate reflects the balance between two competing terms, namely the principal eigenvalue of the Dirichlet Laplacian and an energy functional defined over a certain compactification developed earlier in \cite{MV14}. For $d\geq 3$ and high temperature,  the underlying Gaussian field is also shown to attain very high values under the GMC -- that is, all paths are ``GMC -thick" in this regime. Both statements are natural infinite dimensional extensions of similar behavior captured by $2d$ Liouville quantum gravity and reflect a certain ``atypical behavior" of the GMC: while the GMC volume decays exponentially fast uniformly over all paths, the field itself attains atypically large values on all paths when sampled according to the GMC.  It is also shown that, despite the exponential decay of volume for any temperature, for small enough temperature, the normalized overlap of two independent paths tends to follow one of only a finite number of independent paths for most of its allowed time horizon, allowing the GMC probability to accumulate most of its mass along such trajectories.  
}
\end{quote}

\section{Introduction and the main results.}

\subsection{Main results: informal description}\label{sec-result-informal}

Consider a centered Gaussian field $\{H(\omega)\}_{\omega\in \mathscr C}$ on a complete probability space $(\mathcal E, \mathcal F, \mathbf P)$, with the field  parametrized by a metric space $\mathscr C$ carrying a reference measure $\mu$. For any parameter $\gamma>0$,  the transformed measure  
 \begin{equation}\label{GMC0}
M_\gamma(\d\omega)= \exp\Big\{\gamma H(\omega)- \frac 12 \gamma^2 \E\big[H(\omega)^2\big]\Big\} \,\, \mu(\d\omega), \quad \gamma>0,
\end{equation}
on $\mathscr C$ is known as a {\it Gaussian multiplicative chaos} (GMC) which was constructed by Kahane in a fundamental work \cite{K85}. 
In a general set up, the random field $H(\cdot)$ might include oscillations, which are however cancelled out 
after integration w.r.t. suitable test functions -- that is, $\{H(\omega)\}_{\omega\in \mathscr C}$ could be interpreted as a distribution.
However, these oscillations (when intrinsically present) become highly magnified after exponentiating the field, so that a suitable mollification or a cut-off procedure becomes necessary to interpret the formally defined object \eqref{GMC0} on a rigorous level. The results of the current article are strongly motivated by 
the ongoing investigations on {\it finite dimensional GMC} where the underlying object $H(\cdot)$ is a Gaussian free field in the complex plane leading to a GMC known as the
{\it Liouville quantum gravity}, a random surface that appears as the scaling limit of random planar maps and exhibits remarkable geometric properties in terms of its multi-fractal spectrum and (non-integer) volume decay exponents, see Section \ref{sec-motivation} for a discussion drawing on analogies to the present results. 

\smallskip

We investigate  GMC measures in an {\it infinite dimensional} setting from a geometric viewpoint for the first time in the present article, whose results    
can be summarized as follows. The base space $\mathscr C_T=C([0,T];\R^d)$
stands for the metric space of continuous functions $\omega:[0,T]\to \R^d$ equipped with the uniform metric. This space is tacitly endowed with the Wiener measure $\bP_x$ corresponding to $\R^d$-valued Brownian path starting at $x\in \R^d$. The Gaussian field $\mathscr H_T(\cdot)$ we are interested in is driven by Gaussian space-time white noise ${\dot B}$ (see Section \ref{sec-results} for a precise definition) integrated against the Brownian path:
\begin{equation}\label{H0}
\mathscr H_T(\omega)= \int_0^T\int_{\R^d} \kappa(\omega_s-y) {\dot B}(s,y)\d y \d s, \quad \mbox{so that} \quad \E[\mathscr H_T(\omega)\mathscr H_T(\omega^\prime)]=\int_0^T (\kappa\star\kappa)(\omega_s-\omega^\prime_s) \d s.
\end{equation}
Here the first integral above is interpreted in the It\^o sense (with $\kappa$ being any smooth, spherically symmetric and compactly supported function with $\int_{\R^d} \kappa(x) \d x=1$) and $\E$ denotes expectation w.r.t. the noise ${\dot B}$. Thus, the ambient field $\{\mathscr H_T(\omega)\}_{\omega\in\mathscr C_T}$ is now indexed by Wiener paths 
$\omega\in\mathscr C_T$, and analogous to \eqref{GMC0}, the {\it renormalized} GMC probability measure (at cut-off level $T$) is given by  
\begin{equation}\label{hP0}
\hP_{\gamma,T}(\d\omega)= \frac 1 {\mathscr Z_{\gamma,T}} \, \exp\Big\{ \gamma \mathscr H_T(\omega) - \frac 12\gamma^2 T(\kappa\star\kappa)(0)\Big\} \bP_0(\d\omega),\end{equation}
where $\mathscr Z_{\gamma,T}$ is the total mass of the above exponential weight (under $\bP_0$) and $\gamma>0$ is a parameter known as {\it inverse temperature}. Our first main result, stated formally in Theorem \ref{Theorem 2 new}, quantifies the {\it volume decay exponents} of balls in $\mathscr C_T$ under $\hP_{\gamma,T}$: if $\mathscr N_{r,T}(\varphi)$ denotes a neighborhood of radius $r>0$ around $\varphi\in \mathscr C_T$, then for any $d\geq 1$, $r>0$ and $\gamma>0$, it holds that 
\begin{equation}\label{result0}
\sup_{\varphi\in \mathscr C_T}\hP_{\gamma,T}[\mathscr N_{r,T}(\varphi)]  \lesssim \exp[-\Theta T] \qquad \P-\mbox{a.s.},
\end{equation}
as $T\to\infty$, and for some deterministic and explicit constant $\Theta>0$, which involves two competing terms, namely the principal eigenvalue $\lambda_1>0$ of the Dirichlet Laplacian $-\frac 12 \Delta$ (which captures the volume decay rate that under the base measure $\bP_0$, i.e., when $\gamma=0$), and an additional energy functional minimized over (probability measures on) the translation-invariant compactification constructed in \cite{MV14}, see \eqref{C1}. For $d\geq 3$, for $\gamma>0$ sufficiently small, the decay rate 
$\Theta$ simplifies and involves only the principal eigenvalue $\lambda_1$ and a penalization term $\gamma^2 (\kappa\star\kappa)(0)/2$. A similar exponential lower bound on the GMC probability is also shown to hold pointwise. 
In the second part of Theorem \ref{Theorem 2 new} it is shown that 
\begin{equation}\label{result1}
\begin{aligned}
\frac 1{T(\kappa\star\kappa)(0)} \int \mathscr H_T(\omega) \hP_{\gamma,T}(\d\omega)
 \to  \gamma >0, \qquad \P- \mbox{a.s. when }\gamma\in(0,\gamma_1), d\geq 3,
\end{aligned}
\end{equation}
i.e., the field $\mathscr H_T(\omega)$ attains high values averaged w.r.t. $\hP_{\gamma,T}$. 
Recall  that $\mathrm{Var}^{\mathbf P}[\mathscr H_T]= T(\kappa\star \kappa)(0)$. 
The statements \eqref{result0} and \eqref{result1}, combined together, confirm the following intuition on the {\it atypical behavior} of our GMC measure. Since the former statement implies that the GMC volume decays exponentially uniformly over all paths for {\it any} $\gamma>0$, 
 the only reason that $\hP_{\gamma,T}$ could be non-trivial (in the limit $T\to\infty$) is if there are enough paths where the Gaussian field $\mathscr H_T$ is {\it atypically large}. Thus, it is conceivable that the GMC measure in some sense is ``carried by" sufficiently many such thick paths, and it is natural to wonder what their {\it level of thickness} could be. The second statement \eqref{result1} in Theorem \ref{Theorem 2 new} confirms this intuition -- when $\gamma>0$ remains sufficiently small and $d\geq 3$, all paths are $\hP_{\gamma,T}$-thick, with $\gamma (\kappa\star\kappa)(0)$ being the required level of thickness; see \ref{eq-thick} in Theorem \ref{Theorem 2 new}. 
Similar results pertinent to GMC measure arising from a multiplicative-noise stochastic heat equation in $d\geq 3$ is derived in Corollary \ref{Theorem 3}. 
To complete the picture, we also show that, even though GMC-volume decays exponentially for any temperature $\gamma>0$, once we tune $\gamma>0$ sufficiently {\it large}, the normalized overlap of two independent paths, sampled according to $\hP_{\gamma,T}$,  tends to follow one of only a finite number of independent paths for most of its allowed time horizon, allowing the GMC probability 
to accumulate most of its mass along such trajectories, see Theorem \ref{localization via overlaps} for a precise statement (which holds when $\gamma$ is at least as large as 
$\gamma_1$; and its complementary phase $\gamma\in (0,\gamma_1)$ in $d\geq 3$ is the regime where all paths are GMC-thick, as underlined by \eqref{result1}, see also Remark \ref{remark5}). Let us now explain the background which motivated the current work.

\subsection{Motivation: geometry of the Liouville measure.}\label{sec-motivation} 

 In a finite dimensional setting, GMC measures share close connection to two-dimensional Liouville quantum gravity which 
has seen a lot of revived  interest in the recent years, see \cite{BP21} for an exposition. In this setting, the GMC measure (or the so-called {\it Liouville measure}) appears as the limit $\mu_{\gamma,\eps}(\d x)= \eps^{\gamma^2/2}\exp\{\gamma h_\eps(z)\}\d z$, with $h_\eps$ being a suitable approximation (e.g. circle average) of the (log-correlated) two-dimensional Gaussian free field (GFF) with $\mathrm{Var}(h_\eps)=\log(1/\eps)+ O(1)$ and $\d z$ stands for the Lebesgue measure. A rigorous construction of $\lim_{\eps\to 0} \mu_{\gamma,\eps}$ has been carried out in \cite{K85,RV10,S14,DS11,Ber17} and it is shown that  when $\gamma \in (0, 2)$, $\mu_{\gamma,\eps}$ converges toward a nontrivial measure $\mu_\gamma$ which is diffuse and is known as the {\it{subcritical GMC}}; when $\gamma\geq 2$, we refer to \cite{DRSV14-I,DRSV14-II,MRV16} for the notion of atomic structure of (super-) critical GMC.

Informally, the Liouville measure is a random surface carrying a Riemannian metric tensor (formally given by the exponential of the GFF)
and a parametrization by a domain which preserves the inherent conformal invariance, but distorts the resulting metric and the volume. 
This distortion also reflects a certain ``atypical behavior" of GMC measures: note that the weight 
$\exp\{\gamma h_\eps(z)- \frac 12 \gamma^2 \E[h_\eps(z)^2]\}$ vanishes as $\eps\to 0$ for each fixed $z$ (on some domain $D\subset \R^2$), so 
a non-trivial limiting Liouville measure  must be ``carried" by sufficiently many thick points $z$, that is those $z$ where the
 field $h_\eps(z)$ is {\it atypically large}. In fact, if $\gamma\in (0,2)$ and $z\in D$ is sampled according to $\mu_\gamma$, then 
$\frac {h_\eps(z)}{\log(1/\eps)} \to \gamma >0$. Thus, \eqref{result1} is the infinite-dimensional analogue of the latter assertion (recall that for GFF, $\mathrm{Var}(h_\eps)=\log(1/\eps)+ O(1)$, while for our Gaussian field defined in \eqref{H0}, $\mathrm{Var}^{\mathbf P}(\mathscr H_T)=T(\kappa\star\kappa)(0)$).

 Another prominent aspect of the Liouville measure, which has far reaching consequences,
  hinges on its {\it multi-fractal spectrum}$^{\dagger}$\footnote{$^{\dagger}$For the Liouville measure $\mu_\gamma$ it is known that for any $q\in\R$, $\mathbf E^{\mathrm{GFF}}[\mu_\gamma(B_r(z))^q] \sim r^{\dot B}$ where ${\dot B}= q(2+ \frac{\gamma^2}2)- \frac {\gamma^2 q^2}2$, which is a non-linear function of $q$, is called the multi-fractal spectrum of the Liouville measure.}and {\it volume decay exponents} on microscopic balls. 
Indeed, if $z$ is sampled from the (renormalized) Liouville probability measure $\mu_\gamma$, the (Liouville) scaling exponent $\Delta$ of a set $A\subset \R^2$ is 
defined via the volume decay 
\begin{equation}\label{scalingLiouville}
(\mathbf P^{\ssup{\mathrm{GFF}}}\otimes \mu_\gamma)\big[ N_{\eps}(z) \cap A \ne \emptyset\big] \sim \exp[- (\log 1/\eps) \Delta] \qquad\mbox{ as }\eps\downarrow 0,
\end{equation}
where $N_\eps(z)$ is a neighborhood in the Liouville metric \cite{GM19}. Thus, 
the scaling exponent $\Delta$ allows one to link the volume of a set in the Euclidean geometry to the same in the Liouville geometry$^{\dagger\dagger}$\footnote{$^{\dagger\dagger}$
The scaling exponent is a key object that also appears in the celebrated Knizhnik-Polyakov-Zamolodchikov (KPZ) formula \cite{KPZ88} which dictates that if a 
subset $A$ which is independent of the GFF, has Euclidean scaling exponent $x$ (meaning $P(A\cap B(z,\eps) \ne \emptyset)\sim \eps^x$, with $B(z,\eps)$ being the Euclidean ball of radius $\eps$ around $z$), then $A$ has Liouville scaling exponent $\Delta$; with $x$ and $\Delta$ being related by the quadratic relation $x= \frac {\gamma^2} 4 \Delta^2 + (1-\frac {\gamma^2}4) \Delta$, see \cite{RV11,DS11,BGRV16}.} From this viewpoint, the first assertion \eqref{result0} in Theorem \ref{Theorem 2 new} 
could be interpreted as an infinite dimensional analogue (in an almost sure sense) of \eqref{scalingLiouville} with the constant $\Theta>0$ in \eqref{result0} linking
 the volume under the Wiener measure to the same under the GMC measure $\hP_{\gamma,T}$.

It seems natural to wonder if (some of) the results sketched above can be extended to an infinite-dimensional set up where
important assumptions like logarithmic-covariance, conformal invariance and domain Markov property of $2d$ GFF are no longer available. A canonical choice for infinite-dimensional GMC is the (path) space of continuous functions carrying the Wiener measure. This incentive is also natural  because of close connections of this set up to many disordered systems of statistical mechanics like spin glasses, directed polymer in a random environment as well as to multiplicative-noise stochastic PDEs (as will be explained later). However, a major thrust of the works on these topics over the last years came from studying the total mass of the GMC measure or investigating diffusivity (resp. localization) of the distribution of the Brownian endpoint under the polymer measure. The novel contribution of the present work is the direct investigation of the actual GMC measure (resp. the polymer measure)$^{\dagger\dagger\dagger}$\footnote{$^{\dagger\dagger\dagger}$From the viewpoint of directed polymers, results of the article are new, to the best of our knowledge.}itself from a geometric viewpoint which is inspired by analogous behavior of $2d$ Liouville measure as outlined in the discussion above. 
We now turn to precise statements of the  results announced in Section \ref{sec-result-informal}.

\section{Main results.}

\subsection{Exponential volume decay in GMC-space.}\label{sec-results}

We fix any spatial dimension $d\geq 1$,  write $\mathscr C_T= C([0,T];\R^d)$  for the metric space of continuous functions $\omega: [0,T] \to \R^d$ equipped with the uniform norm 
\begin{align}
\label{definition of infty,T norm}
	\lVert\omega\rVert_{\infty, T}=\sup_{s\in[0,T]}|\omega(s)|, \quad\mbox{with}\,\, \mathscr N_r(\omega)= \mathscr N_{r,T}(\omega)= \big\{\varphi \in \mathscr C_T\colon \|\varphi- \omega\|_{\infty,T} < r\big\}.
	\end{align} 
$\mathscr C_T$ is tacitly equipped with the Wiener measure $\bP_x$ corresponding to an $\R^d$-valued Brownian motion starting at $x\in \R^d$. For any $t>0$, $\mathcal G_t$ will stand for the $\sigma$-algebra 
generated by the path $(\omega_s)_{0\leq s \leq t}$ until time $t$. Let $(\mathcal E,\mathcal F, \mathbf P)$ be a complete probability space and ${\dot B}$ denotes a Gaussian space-time white noise, which is independent of the Brownian path defined above. 
In other words, if $\mathcal S(\R_+\times \R^d)$ denotes the space of rapidly decreasing Schwartz functions, $\{{\dot B}(f)\}_{f\in \mathcal S(\R_+\times \R^d)}$ is a centered Gaussian process with covariance $\mathbf E[ {\dot B}(f)\,\, {\dot B}(g)]= \int_0^\infty \int_{\R^d} f(t,x) g(t,x) \d x \d t$ for $f, g \in \mathcal S(\R_+\times \R^d)$, with 
$\E$ denoting expectation w.r.t. $\P$.

We also fix a nonnegative function $\kappa$ which is smooth, spherically symmetric and is supported in a ball $B_{1/2}(0)$ of radius $1/2$ around $0$ and normalized to have total mass $\int_{\R^d} \kappa(x)\d x=1$. For any fixed Brownian path $\omega\in \mathscr C_T$ we define the It\^o integral 
\begin{equation}
\mathscr H_T(\omega)= \int_0^T \int_{\R^d} \kappa(\omega_s- y) {\dot B}(s,\d y) \d s,\quad\text{with }\quad  \E[\mathscr H_T^2(\omega)]= T (\kappa\star\kappa)(0).
\end{equation}
For any $\gamma>0$,  the resulting {\it (renormalized) Gaussian multiplicative chaos} is a probability measure on the space $\mathscr C_T$ defined as
\begin{equation}\label{MT}
\begin{aligned}
&\hP_{\gamma,T}(\d\omega)= \frac 1 {\mathscr Z_{\gamma,T}} \, \exp\bigg(\gamma \mathscr H_T(\omega) - \frac {\gamma^2T}2  (\kappa\star\kappa)(0)\bigg) \bP_0(\d\omega), \quad\mbox{with   }
\mathscr Z_{\gamma,T}= \bE_0\big[\e^{\gamma \mathscr H_T(\omega) - \frac {\gamma^2T}2  (\kappa\star\kappa)(0)}\big].
\end{aligned}
\end{equation}

To define the decay rate $\Theta$ appearing in \eqref{result0}, we need some further definitions.  
We denote by $\Mcal_1= {\Mcal_1}(\R^d)$ (resp. $\Mcal_{\leq 1}$) the space of probability (resp. sub-probability) measures on $\R^d$, 
which acts as an additive group of translations on these spaces. 
Let $\widetilde\Mcal_1= \Mcal_1 \big/ \sim$ be the quotient space 
of $\Mcal_1$ under this action, that is,  for any $\mu\in \Mcal_1$, its {\it{orbit}} is defined by $\widetilde{\mu}=\{\mu\star\delta_x\colon\, x\in \R^d\}\in \widetilde\Mcal_1$. The quotient space $\widetilde\Mcal_1$ can be embedded in a larger space 
\begin{equation*}
\X: =\bigg\{\xi:\xi=\{\widetilde{\alpha}_i\}_{i\in I},\alpha_i\in \mathcal{M}_{\leq 1},\sum_{i\in I}\alpha_i(\R^d)\leq 1\bigg\}
\end{equation*}
which consists of all empty, finite or countable collections of orbits from $\widetilde\Mcal_{\leq 1}$ whose masses add up to at most one. The space $\X$ and a metric 
structure there was introduced in \cite{MV14}, and it was shown that under that metric, $\widetilde\Mcal_1(\R^d)$ is densely embedded in $\X$ and any sequence in 
$\widetilde\Mcal_1(\R^d)$ converges along some subsequence to an element $\xi$ of $\X$ -- that is, $\X$ is the {\it compactification} of the quotient space $\widetilde\Mcal_1(\R^d)$, 
see Section \ref{spaceX} for details. On the space $\X$ we define an energy functional  
\begin{equation}\label{def-Phi}
\begin{aligned}
&F_\gamma(\xi)=\frac{\gamma^2}{2}\sum_{\widetilde\alpha\in\xi}\int_{\R^d\times\R^d}(\kappa\star\kappa)(x_1-x_2)\prod_{j=1}^2\alpha(\mathrm{d}x_j) \quad \forall \, \xi\in \X, \quad\mbox{and}\\
&{\mathcal E}_{F_\gamma}(\vartheta)= \int_{\X} F_\gamma(\xi)\,\vartheta(\d\xi) \qquad \vartheta\in\Mcal_1(\X).
\end{aligned}
\end{equation}
Here $\Mcal_1(\X)$ denotes the space of probability measures on the space $\X$.  There is an interesting connection between the structure of the space $\X$ and the solution of the variational problem $\sup_{\vartheta\in\Mcal_1(\X)}{\mathcal E}_{F_\gamma}(\vartheta)$: Indeed, there is a non-empty, compact subset ${\mathfrak m_\gamma}\subset \Mcal_1(\X)$
consisting of the maximizer(s) of the variational problem $\sup_{\vartheta\in{{\mathfrak m_\gamma}_\gamma}}{\mathcal E}_{F_\gamma}(\vartheta)$, the maximizing set is a singleton $\delta_{\tilde 0}\in \X$ for $d\geq 3$ and small enough $\gamma$, and any maximizer assigns positive mass only to those elements of the compactification $\X$ whose total mass add up to one, see Proposition \ref{prop-m} for details. We are now ready to state our first main result.$^\ddagger$\footnote{$^\ddagger$Here and in the sequel, $\lambda_1(r)$ will denote for the principal eigenvalue of $-\frac 12 \Delta$ on a ball of radius $r$ around the origin with Dirichlet boundary condition.}


\begin{theorem}[Scaling exponents of GMC in the Wiener space]\label{Theorem 2 new}
Fix any $d\in \N$, $\gamma>0$ and $r> 0$.
\begin{itemize}
\item[(1)] $\P$-almost surely, 
\begin{align}
&\limsup_{T\to\infty}\sup_{\varphi\in\mathscr C_T}\frac{1}{T}\log\hP_{\gamma,T}\big[\mathscr N_{r}(\varphi)\big]\leq -\Theta , \quad \mbox{where} \label{esti}
\\
&\Theta: = \frac 14 \lambda_1(\sqrt 2 r) - \frac {\gamma^2 (\kappa\star\kappa)(0)}2 + \bigg[\frac 12 \sup_{\vartheta\in \mathfrak m_{2\gamma}} {\mathcal E}_{F_{2\gamma}}(\vartheta) 
- \sup_{\vartheta\in \mathfrak m_{\gamma}} {\mathcal E}_{F_{\gamma}}(\vartheta)\bigg].\label{C1}
\end{align}
Moreover, for any $d\in \N$ and $\gamma>0$, there exists $r_0=r_0(d,\gamma)$ such that for $r\in (0,r_0)$, $\Theta>0$. Conversely, for any $d\in \N$ and $r>0$,
there is $\gamma_c>0$ such that for $\gamma \in (0,\gamma_c)$, $\Theta>0$. Finally, for $d\geq 3$ and $\gamma>0$ sufficiently small, 
$$
\Theta= \frac 14\lambda_1(\sqrt 2r) - \frac{\gamma^2 (\kappa\star\kappa)(0)}2>0.
$$ Finally, for any $d\in \N, \gamma>0, r>0$, an exponential lower bound similar to \eqref{esti} also holds pointwise. 
\item[(2)] Fix $d\geq 3$ and $\gamma>0$ sufficiently small. Then, 
\begin{equation}\label{eq-thick}
\lim_{T\to\infty} \bE^{\hP_{\gamma,T}}\bigg[ \frac{\mathscr H_T(\cdot)}T\bigg] = \gamma (\kappa\star\kappa)(0)>0 \qquad \P-\,\mbox{a.s.}
\end{equation}
\end{itemize}
\end{theorem}

\begin{remark}[{\it Kahane's GMC and Theorem \ref{Theorem 2 new}}.]\label{rmk-Kahane}
As remarked earlier, the statements \eqref{esti}-\eqref{eq-thick} imply exponential volume 
decay, underline the emergence of thick paths and determine the level of their thickness; recall the discussion below \eqref{result1} in Section \ref{sec-result-informal}. It would be quite intriguing to prove existence of the infinite-volume limit $\lim_{T\to\infty}\hP_{\gamma,T}(\cdot)$ in the full uniform integrability/ weak disorder phase and study properties of this limit as done by Kahane \cite{K85} for log-correlated fields. 
It is reasonable to expect the results of Theorem \ref{Theorem 2 new} to be helpful in this pursuit. Indeed, the existence of the $2d$ Liouville measure, already in the $L^2$-phase $\gamma \in (0,\sqrt 2)$ requires good estimates for the $L^2$-norm $\mathbf E[(\mu_{\gamma,\eps}(N)- \mu_{\gamma,\eps/2}(N))^2]$ under the 
{\it approximating} Liouville measure $\mu_{\gamma,\eps}$ (e.g. defined via the circle average of the GFF), see the construction of Berestycki \cite{Ber17}. Extending the limit to the uniform integrable phase $\gamma \in (0,2)$ requires additional information on thick points as one needs to remove points that are thicker than a prescribed level of typical thickness (again see \cite{Ber17}). However, because of the present infinite dimensional set up (and lack of local compactness), additional difficulties are likely to arise, tackling which seems to go beyond the scope of the current article.$^{\star}$\footnote{$^{\star}$One such difficulty might already arise from the covariance structure of the underlying field $\mathrm{Cov}[\mathscr H_T(\omega) \mathscr H_T(\omega^\prime)]= \int_0^T (\kappa\star \kappa)(\omega_s-\omega^\prime_s) \d s$ which depends on the mollification scheme $\kappa$ of the white-noise field ${\dot B}$. In contrast, note that $\mathrm{Var}(h_\eps(z))=\log(1/\eps)+ R(z,D)$ (with $R(z,D)$ being the conformal radius of the simply connected domain $D$ viewed from $z$) does not depend on any mollification scheme if we consider the circle average approximation $h_\eps$ of the $2d$ GFF $h$.}\qed
 \end{remark}

It is also useful to record the following reformulation of the above results concerning a GMC measure arising from the solution of the multiplicative noise {\it stochastic heat equation} in $d\geq 3$, which can be written as an It\^ o SDE
\begin{equation}\label{she}
\d u_{\eps}(t,x)= \frac 12 \Delta u_{\eps}(t,x)+ \gamma(\eps, d) u_{\eps}(t,x) {\dot B}_{\eps}(t,x)\d t, \quad\mbox{with  } \gamma(\eps,d)= \gamma \eps^{\frac{d-2}2},
\end{equation}
and with the initial condition is $u_\eps(0,x)=1$. Here ${\dot B}_\eps(t,x)= ({\dot B} \star \kappa_\eps)(t,x)=\int_{\R^d} \kappa_\eps(x-y) {\dot B}(t,y) \d y$ is a (spatially) mollified noise in $d\geq 3$
and $\kappa_\eps=\eps^{-d}\kappa(x/\eps)$ with $\kappa$ as before, so that $\int_{\R^d} \kappa_\eps(x) \d x=1$ and $\kappa_\eps(x)\d x \Rightarrow \delta_0$ weakly as probability measures as $\eps\to 0$.  By Feynman-Kac formula, the solution to \eqref{she} is given by
\begin{equation}\label{FK-SHE}
u_\eps(t,x)= \bE_x\bigg[\exp\bigg\{ \gamma(\eps,d) \int_0^{t} \int_{\R^d} \kappa_\eps\big(\omega_{s}- y\big) \, {\dot B}(t- s, \d y) \d s - \frac {\gamma(\eps,d)^2 t } 2  (\kappa_\eps\star\kappa_\eps)(0)\bigg\}\bigg].
\end{equation}
This solution (resp. the Cole-Hopf solution $h_\eps:= \log u_\eps$ of the \textit{Kardar-Parisi-Zhang} equation in $d\geq 3$)
is directly linked with the total mass $\mathscr Z_{\gamma,T}$ (resp. the free energy $\log \mathscr Z_{\gamma,T}$) of the GMC measure \eqref{MT}, and this link was first observed and used in \cite{MSZ16} (see also \cite{M17,CCM19,CCM19-II,BM19,BM19-II,CNN20,LZ20} for further progress). The above representation \eqref{FK-SHE}  leads to the renormalized GMC probability measure 
\begin{equation}\label{Mbar}
\begin{aligned}
\overline{\mathscr M}_{\gamma,\eps,t}^{\ssup x}(\d\omega) 
&= \frac 1 {\mathscr Z_{\gamma,\eps,t}} \exp\bigg\{\gamma\eps^{\frac{d-2}2} \int_0^t \int_{\R^d} \kappa_\eps\big(\omega_{s}- y\big) \, {\dot B}(t- s, \d y) \d s  
- \frac {\gamma^2 \eps^{d-2} t } 2  (\kappa_\eps\star\kappa_\eps)(0)\bigg\} \bP_x(\d\omega)
\end{aligned} 
\end{equation}
with $\mathscr Z_{\gamma,\eps,t}$ being the normalizing constant, as usual. 
Here is our next result, for which we will also write $\overline{\mathscr M}_{\gamma,\eps,t}=\overline{\mathscr M}_{\gamma,\eps,t}^{\ssup 0}$.
\begin{cor}[Exponential decay of GMC corresponding to SHE in $d\geq 3$]\label{Theorem 3}
 Fix $d\geq 3$. Then for any $\gamma>0$ and $t>0$, 
	$$
	\limsup_{\eps\downarrow 0} \sup_{\varphi\in\mathscr C_{t}} \eps^2 \log\overline{\mathscr M}_{\gamma, \eps, t}\big[\mathscr N_{\eps,t}(\varphi)\big]\leq (- \Theta t) 
	\qquad \mbox{in }\P-\,\mbox{probability},
	$$
	with $\Theta>0$ determined in Theorem \ref{Theorem 2 new}, and with $\mathscr N_{\eps,t}(\varphi)=\{\omega\in \mathscr C_t\colon \|\omega-\varphi\|_{\infty,t} \leq \eps\}$.	
	\end{cor}

Let us now turn to the regime when $\gamma$ is chosen to be large. In this regime, 
for localization results on the {\it endpoint distributions} ${\mathbb Q}_{\gamma,t}=\hP_{\gamma,t}[\omega_t\in \cdot]$
we refer to \cite{CSY03,V07,CC13,BC16,BM19-II}. For the discrete lattice, for log-correlated Gaussian fields (e.g. for $2d$ discrete GFF) it has been shown in \cite{AZ14, AZ15,MRV16} that for large $\gamma$, the normalized covariance of two points sampled from the Gibbs measure is either $0$ or $1$ and the joint distribution of the Gibbs weights converges in a suitable sense to that of a Poisson-Dirichlet variable and similar results can be found in \cite{BC19} for general Gaussian disordered systems in the lattice setting. 
The following result extends such statements concerning overlap localization to the current GMC set up. 
\begin{theorem}
	\label{localization via overlaps}
	Let $\mathrm{Cov}_T(\omega,\omega^\prime)= \frac 1 {T (\kappa\star\kappa)(0)}\mathbf E[\mathscr H_T(\omega)\mathscr H_T(\omega^\prime)]$ (recall \eqref{H0}). 
	Fix $\gamma>0$ such that $\gamma$ is a point of differentiability of $\lambda(\gamma):= \inf_{\vartheta\in{\mathfrak m_\gamma}}[\frac{\gamma^2 (\kappa\star\kappa)(0)}2 - {\mathcal E}_{F_\gamma}(\vartheta)]$ (recall \eqref{def-Phi}) and moreover assume that $\lambda^\prime(\gamma)<\gamma (\kappa\star\kappa)(0)$. Then for every $\eps> 0$ there exist $\delta,T_0>0$ and an integer $k\in\N$ and $\omega^{\ssup{1}},...,\omega^{\ssup{k}}\in\mathscr C_\infty$ such that
	$$
	\P\bigg[ \hP_{\gamma,T} \bigg(\bigcup_{i=1}^k \mathrm{Cov}_T(\omega^{\ssup i}, \omega^{\ssup{k+1}}) \geq \delta\bigg) \geq 1-\eps\bigg]\geq 1-\eps 
	$$ 
	for all $T\geq T_0$.
\end{theorem}

\begin{remark}\label{remark5}
	Assuming that $\lambda(\gamma)$ is differentiable at $\gamma$, we always have $\lambda^\prime(\gamma) \leq \gamma (\kappa\star\kappa)(0)$ and 
	actually for $\gamma<\gamma_1$ (with $\gamma_1$ from Proposition \ref{prop-m}), we have $\lambda^\prime(\gamma)=\gamma (\kappa\star \kappa)(0)$ (cf. Proposition \ref{prop3}, part (ii)). However, the requirement concerning the strict bound  $\lambda^\prime(\gamma)<\gamma (\kappa\star\kappa)(0)$ in Theorem \ref{localization via overlaps} is related to $\gamma$ being large, at least as large as $\gamma_1$. 
	Therefore, if $\gamma$ is sufficiently large such that Theorem \ref{localization via overlaps} holds, we necessarily have $\gamma\geq \gamma_1$ -- that is, Theorem \ref{localization via overlaps} can not hold if $\gamma \in (0,\gamma_1)$ in $d\geq 3$, which is the regime in which $\mathscr H_T$ attains large values under $\hP_{\gamma,T}$, as shown in Theorem \ref{Theorem 2 new}, part (2).\qed
\end{remark}

\subsection{Outline of the proofs.}\label{sec:outline}
For the convenience of the reader, we briefly describe the main ingredients of the proof. 
Showing the estimate \eqref{esti} and the identity \eqref{eq-thick} in Theorem \ref{Theorem 2 new} rely on a large deviation route combined with tools from Malliavin calculus. First, 
the free energy $\log \bE_0[\e^{\gamma \mathscr H_T} \1_A]$ on the path space is decomposed using It\^o calculus into a martingale and an integral functional $F(\mu_t)= \int\int V(x-y) \mu_t(\d x) \mu_t (\d y)$ of the GMC distribution $\mu_t:=\hP_t[\omega_t \in \cdot]$. The first key observation is that, while $\mu_t$ is itself a probability measure on $\R^d$, the function $ F(\cdot)$ depends {\it only} on its {\it orbit} $\widetilde\mu_t=(\mu_t\star\delta_x)_{x\in \R^d}$. These orbits are elements of the quotient space $\widetilde\Mcal_1= \Mcal_1/\sim$ of probability measures $\mu\in \Mcal_1=\Mcal_1(\R^d)$ (here $\mu\sim \nu$ if and only if $\nu=\mu\star \delta_x$ for some $x\in \R^d$). However, since both spaces $\Mcal_1$ as well as its quotient $\widetilde\Mcal_1$ are non-compact in the (quotient) weak topology, the functional $F$ does not retain any continuity property (even when it is extended to $\widetilde\Mcal_1$). For the same reasons, a Markovian semigroup evolved by $\mu_t$ (which was constructed in \cite{BC16} for proving localization of discrete directed polymers$^{\star\star}$\footnote{$^{\star\star}$This approach also used the topology of \cite{MV14} but worked with a different (but equivalent) metric which works well in the discrete lattice. Presently in the continuum. we employ the metric directly from \cite{MV14}.}) also does not yield an invariant measure in $\widetilde\Mcal_1$. This is where the construction of \cite{MV14} becomes useful: the quotient $\widetilde\Mcal_1(\R^d)$ can be embedded in the compactified space $\X$, which, in contrast to $\widetilde\Mcal_1$, is now sufficiently large to contain {\it all} limit points of the GMC distribution $\mu_t \in \Mcal_1(\R^d)$ (but still sufficiently small so that the quotient $\widetilde\Mcal_1(\R^d)$ is densely embedded in $\X$)$^\diamond$\footnote{$^\diamond$For example, let $\mu_n$ be a Gaussian mixture $\frac 13 N(n,1)+ \frac 13 N(-n,1) + \frac 13 N(0,n)$. While neither $\mu_n$, nor any of its components, has a weak limit in $\Mcal_1(\R^d)$, the sequence $(\widetilde\mu_n)_n\subset \widetilde{\Mcal}_1(\R^d)\hookrightarrow \X$ converges (in the topology of $\X$) along a subsequence 
to the element $\xi=(\widetilde\alpha_1,\widetilde\alpha_1)\in \X$, where $\widetilde\alpha_1$ is the orbit any Gaussian (i.e., with arbitrary mean)  with variance $1$.}. By construction,  one can then  ``lift" the above functional $F(\cdot)$ (recall \eqref{def-Phi}) as well as the Markovian dynamics on this compactification $\X$. The topology of $\X$ (resulting from leveraging the metric directly from \cite{MV14}) then guarantees an invariant measure of the semigroup. Together with the required continuity properties of the aforementioned functionals on $\X$, we then deduce the desired exponential decay of the GMC volume $\sup_{\varphi\in \mathscr C_T} \hP_T[\mathscr N_r(\varphi)]$, with a decay rate which is given by the variational formula on the compactified space $\X$ which always admits minimizer(s) in $\Mcal_1(\X)$ (this minimizer is unique when $\gamma>0$ is small and $d\geq 3$ which further simplifies the variational formula).  The existence of thick points \eqref{eq-thick} then follow from the above arguments together with some tools from Malliavin calculus. 
It should be mentioned that the current approach  does not rely on sub-additivity arguments which have been previously used as a powerful tool in obtaining deep results about the free energy of directed polymers \cite{CY06}). Here, the uniform exponential decay of the GMC volume (for {\it any} temperature $\gamma$) results from rather explicit variational formulas involving functionals over (the space of probability measures on) $\X$. From the viewpoint of the discussion in Section \ref{sec-motivation}),  the decay rates here are the analogues of the scaling exponents of the Liouville measure.

 For proving Theorem \ref{localization via overlaps} we also transport tools from Malliavin calculus to the current set up: 
 define the (infinite dimensional) Ornstein-Uhlenbeck operator $\mathscr L= - \delta \circ D$ on the
 abstract Wiener space $(\mathcal E, \mathcal F, \mathbf P)$ (with $D$ being the Malliavin derivative and $\delta$ being ``divergence" acting as an adjoint of $D$) and deduce a Poincar\'e inequality $\mathbf{Var}\big(\int_0^t \mathscr L f_T(\boldsymbol{\dot B}_r)\d r\big) \leq 2 t \mathbf E\big( \| D f_T({\dot B})\|^2_{L^2([0,T]\otimes \R^d)}\big)$
 for $f_T(\gamma,{\dot B})=\frac 1 T \log Z_{\gamma,T}= \frac 1 T \log \bE_0[\exp\{\gamma \mathscr H_T(\omega,{\dot B})\}]$. Here 
 $\boldsymbol{\dot B}_t(s,x)= \e^{-t} {\dot B}(s,x)+ \e^{-t} \eta((\e^{2t}-1)^{-1}s,x)$
  is a {\it white noise flow}, with $\eta$ being an independent copy of the space-time white noise ${\dot B}$. Applying Chebyshev's inequality, we then get a bound 
  $\P\big[\frac 1 t\int_0^t (\mathscr L f_T)(\boldsymbol{\dot B}_r) \d r >\gamma\eps/2] \leq C (t T\eps^{2})^{-1}$ for all $\eps, t>0$, and together with
 the assumption that $\lambda(\gamma):= \lim_{T\to\infty} f_T(\gamma,{\dot B})$ is differentiable, and $\lambda^\prime(\gamma)<  \gamma (\kappa\star \kappa)(0)$, we obtain a localization property showing that for $t=t(\eps,\gamma)$ sufficiently large, 
 $\liminf_{T\to\infty} \P\big[\frac 1 {t/T}\int_0^{t/T}\!\! \d r\,\,\frac{1}{T} \int_0^T\!\!\d s\,\, \widehat{\bE}_{\gamma,T}^{\ssup{\boldsymbol{\dot B}_r}^{\otimes}} \big[(\kappa\star \kappa)(\omega_s-\omega_s^\prime)\big] \geq \alpha\big] \geq 1- \frac{\eps}2$. This builds on a recent technique \cite{BC19} developed for showing localization of general Gaussian disordered systems (on discrete lattices). Together with the fact $\widehat\bE_{\gamma,T}[|\mathscr H_T-T\lambda^\prime(\gamma)|] = o(T)$ (which holds true under the imposed hypotheses) we then have the required concentration of the covariance in Theorem \ref{localization via overlaps}.

\smallskip 

\noindent{\bf Organization of the rest of the article:} 
The rest of the article is organized as follows: In Section \ref{sec-spaceX} we collect the properties of the space $\X$, those of the relevant energy functions defined on $\X$ and $\Mcal_1(\X)$ and their Malliavin derivatives, while the proofs of Theorem \ref{Theorem 2 new}- Theorem \ref{localization via overlaps} constitute Section \ref{sec-proofs}.

\noindent{\bf Notation:} For convenience, we will adopt the following notation throughout the sequel concerning the GMC measure defined in \eqref{MT}. 
\begin{align}
\begin{split}\label{notation}
	&\widehat{\mathscr M}_T=\widehat{\mathscr M}_{\gamma,T},\quad \mathscr{Z}_T=\mathscr{Z}_{\gamma,T}=Z_{\gamma,T}\,\,\e^{-\frac{\gamma^2}{2}T(\kappa\star\kappa)(0)}, \quad \mbox{with}\quad	
	 Z_T=Z_{\gamma,T}=\bE_0\big[\e^{\gamma\mathscr H_T}\big] \,\,\,\mbox{and}
	 \\
	 	&\qquad\qquad\qquad  V=\kappa\star \kappa.
\end{split}
\end{align}
Also, unless otherwise specified, expectation with respect to the GMC probability measure $\widehat{\mathscr M}_T$ will be written as 
$\widehat{\bE}_T$. For two independent Brownian motions $\omega,\omega^\prime$, we write for the product GMC probability 
\begin{align}
\begin{split}
\label{Mprod}
\hP_T^\otimes(\d \omega,\d \omega^\prime)&=\frac{1}{\mathscr Z_T^2}\exp\big\{\gamma(\mathscr H_T(\omega)+\mathscr H_T(\omega^\prime))-\gamma^2TV(0)\big\}\bP_0^\otimes(\d \omega,\d \omega^\prime)\\
\end{split}
\end{align} 
and expectation with respect to $\hP_T^\otimes$ will be written as $\bET^\otimes$.
Finally, for any $t>0$, $\mathcal G_t$ will stand for the $\sigma$-algebra generated by the Brownian path $(\omega_s)_{0\leq s \leq t}$. For any $A\in\mathcal G_t$, we will write 
\begin{align}\label{definition Z_T(A)}
	Z_t(A)=\bE_0\big[\1_A\,\,\e^{\gamma\mathscr H_t}\big].
\end{align}

\section{The space $\X$, energy functionals and its Malliavin derivatives.}\label{sec-spaceX}

\subsection{The space $\X$.}\label{spaceX}

We denote by $\Mcal_1= {\Mcal_1}(\R^d)$ (resp., $\Mcal_{\leq 1}$) the space of probability (resp., subprobability) distributions on $\R^d$ and by $\widetilde\Mcal_1= \Mcal_1 \big/ \sim$ the quotient space 
of $\Mcal_1$ under the action of $\R^d$ (as an additive group on $\Mcal_1$), that is, for any $\mu\in \Mcal_1$, its {\it{orbit}} is defined by $\widetilde{\mu}=\{\mu\star\delta_x\colon\, x\in \R^d\}\in \widetilde\Mcal_1$. Then we define 
\begin{equation}\label{eq-space-X}
\X=\Big\{\xi:\xi=\{\widetilde{\alpha}_i\}_{i\in I},\alpha_i\in \mathcal{M}_{\leq 1},\sum_{i\in I}\alpha_i(\R^d)\leq 1\Big\}
\end{equation}
to be the space of all empty, finite or countable collections of orbits of subprobability measures with total masses $\leq 1$. Note that the quotient space $\widetilde\Mcal_1(\R^d)$ is embedded in $\X$ -- that is, for any $\mu\in \Mcal_1(\R^d)$, $\widetilde\mu\in\widetilde\Mcal_1(\R^d)$ and the single orbit element $\{\widetilde\mu\}\in \X$ belongs to $\X$ (in this context, sometimes we will write $\widetilde\mu\in \X$ for $\{\widetilde\mu\}\in \X$). 

The space $\X$ also comes with a metric structure. If for any $k\geq 2$, $\mathcal H_k$ is the space of functions $h:\left(\R^d\right)^k\rightarrow \R$ which are invariant under rigid translations and which vanish at infinity, we define for any $h\in \mathcal{H}=\bigcup_{k\geq 2}\mathcal{H}_k$, the functionals
\begin{equation}\label{Lambda-def}
\mathscr J (h,\xi)=\sum_{\widetilde\alpha\in\xi}\int_{(\R^d)^k }h(x_1,\ldots, x_k)\alpha(\d x_1)\cdots\alpha(\d x_k).
\end{equation} 
A sequence $\xi_n$ is desired to converge to $\xi$ in the space $\X$ if 
$$
\mathscr J(h,\xi_n)\to \mathscr J(h,\xi)\qquad \forall \,\, h\in \mathcal H.
$$
This leads to the following definition of the metric $\mathbf D$ on $\X$.
For any $\xi_1,\xi_2\in \X$, we set 
\begin{align*}
\mathbf{D}(\xi_1,\xi_2)&=\sum_{r=1}^{\infty}\frac{1}{2^r}\frac{1}{1+\lVert h_r\rVert_{\infty}} \bigg|\mathscr J(h_r,\xi_1)- \mathscr J(h_r,\xi_2)\bigg| \\
&=\sum_{r=1}^{\infty}\frac{1}{2^r}\frac{1}{1+\lVert h_r\rVert_{\infty}}\bigg|\sum_{\widetilde{\alpha}\in \xi_1}\int h_r(x_1,...,x_{k_r})\prod_{i=1}^{k_r}\alpha(\mathrm{d}x_i)-\sum_{\widetilde{\alpha}\in \xi_2}\int h_r(x_1,...,x_{k_r})\prod_{i=1}^{k_r}\alpha(\mathrm{d}x_i)\bigg|.
\end{align*}

The following result was proved in \cite[Theorem 3.1-3.2]{MV14}. 
\begin{theorem}\label{thm-compact}
We have the following properties of the space $\X$.
\begin{itemize}
\item $\mathbf D$ is a metric on $\X$ and the space $\widetilde\Mcal_1(\R^d)$ is dense in $(\X,\mathbf D)$.
\item Any sequence in $\widetilde\Mcal_1(\R^d)$
has a convergent subsequence with a limit point in $\X$. Thus, $\X$ is the completion and the compactification of the totally bounded metric space $\widetilde\Mcal_1(\R^d)$ under 
$\mathbf D$.
\item Let a sequence $(\xi_n)_n$ in $\X$ consist of a single orbit $\widetilde\gamma_n$ and $\mathbf D(\xi_n,\xi)\to 0$ where $\xi=(\widetilde\alpha_i)_i\in \X$ such that $\alpha_1(\R^d)\geq \alpha_2(\R^d) \geq \dots$.
	Then given any $\eps>0$, we can find $k\in \N$ such that $\sum_{i>k} \alpha_i(\R^d) <\eps$ and we can write 
		$\gamma_n= \sum_{i=1}^k\alpha_{n,i}+ \beta_n$,  
		such that	
		\begin{itemize}
		\item  for any $i=1,\dots,k$, there is a  sequence $(a_{n,i})_n\subset \R^d$ such that
		\begin{equation*}
		\begin{aligned}
		\alpha_{n,i}\star \delta_{a_{n,i}}  \Rightarrow \alpha_i \quad\mbox{with}\quad \lim_{n\to\infty} \, \inf_{i\ne j}\,\, |a_{n,i}- a_{n,j}| =\infty.
		\end{aligned}
		\end{equation*}
		\item The sequence $\beta_n$ totally disintegrates, meaning for any $r>0$, $\sup_{x\in \R^d} \beta_n\big(B_r(x)\big)\to 0$.\qed
	\end{itemize} 
	\end{itemize}
\end{theorem}

\subsection{Energy functionals on $\X$.} 

We now define some key functionals on the space $\X$ which will be quite useful in the present context. First, we set 
 $F_\gamma:\X\rightarrow \R$ to be
\begin{align}\label{functional Phi}
F_\gamma(\xi)=\frac{\gamma^2}{2}\sum_{i\in I}\int_{\R^d\times\R^d}V(x_1-x_2)\prod_{j=1}^2\alpha_i(\mathrm{d}x_j), \qquad \xi=(\widetilde{\alpha}_i)_{i\in I}.
\end{align}
 Because of shift-invariance of the integrand in \eqref{functional Phi}, $F_\gamma$ is well-defined on $\X$. Moreover, we have 
 \begin{lemma}\label{lemma-Phi}
 $F_\gamma$ is continuous and non-negative on $\X$, and $F_\gamma(\cdot)\leq \frac{\gamma^2}{2}V(0)$.
 \end{lemma}
 \begin{proof}
For the continuity of $F_\gamma$, we refer to \cite[Corollary 3.3]{MV14}. Recall that $V=\kappa\star \kappa$ and $\kappa$ is rotationally symmetric. Hence, for any $\alpha\in \Mcal_{\leq 1}(\R^d)$, by Cauchy-Schwarz inequality, 
\begin{equation}\label{function Phi-2}
\begin{aligned}
\int_{\R^{2d}}& V(x_1-x_2) \, \alpha(\d x_1)\, \alpha(\d x_2) =\int_{\R^{2d}} \alpha(\d x_1)\alpha(\d x_2) \, \int_{\R^d} \d z \, \kappa(x_1-z) \, \kappa(x_2-z) \\
&\leq \int_{\R^{2d}} \alpha(\d x_1)\alpha(\d x_2) \, \bigg[\int_{\R^d} \d z \kappa^2(x_1-z)\bigg]^{1/2}  \,\, \bigg[\int_{\R^d} \d z \kappa^2(x_2-z)\bigg]^{1/2} \leq \alpha\big(\R^d\big)^2 \|\kappa\|_2^2. 
\end{aligned}
\end{equation}
Thus, $F_\gamma(\xi)\leq \frac{\gamma^2 (\kappa\star\kappa)(0)}2 \sum_{i\in I} (\alpha_i(\R^d))^2 \leq \frac{\gamma^2 (\kappa\star\kappa)(0)}2$ since for $\xi=(\widetilde\alpha_i)_{i\in I} \in \X$ we have $\sum_{i\in I} \alpha_i(\R^d)\leq 1$. 
Moreover, since $V=\kappa\star\kappa$ is non-negative, also ${F_\gamma}(\cdot)\geq 0$.
\end{proof}

Next, for any $\alpha\in \Mcal_\leq(\R^d)$, let
$$
\mathscr G_t(\alpha)=\int_{\R^d}\int_{\R^d}\alpha(\mathrm{d}z) \bE_z\bigg[\1{\{\omega_t\in\mathrm{d}x\}}\, \exp\bigg\{\gamma\mathscr H_t(\omega)-\frac{\gamma^2}{2}tV(0)\bigg\}\bigg]
$$
and note that for any $a\in \R^d$ and $t>0$, $\mathscr G_t(\alpha_i)\overset{\ssup d}=\mathscr G_t(\alpha_i\star\delta_a)$. Hence, we may define 
$\mathscr G_t, \overline{\mathscr G}_t: \X\to \R$ as 
\begin{align}
\begin{split}\label{scr F}
&\qquad\mathscr G_t(\xi)=\sum_i \mathscr G_t(\alpha_i), \qquad\quad \overline{\mathscr G}_t(\xi)=\mathscr G_t(\xi)+\E\big[\mathscr Z_t-\mathscr G_t(\xi)\big] 
\qquad\forall \xi=(\widetilde\alpha_i)_{i\in I}\in \X \\
\end{split}
\end{align}

Next, for any $t>0$, and for $\xi=(\widetilde\alpha_i)_i \in \X$, we set 
\begin{equation}\label{alpha after time t}
\begin{aligned}
&\alpha_i^{\ssup t}(\mathrm{d}x):=\frac{1}{\overline{\mathscr G}_t(\xi)}\int_{\R^d}\alpha_i (\mathrm{d}z)\bE_z\bigg [\1{\{\omega_t\in\mathrm{d}x\}}\, \exp\big\{\gamma\mathscr H_t(\omega)-\frac{\gamma^2}{2}tV(0)\big\}\bigg] 
\\
&\xi^{\ssup t}:= \big(\widetilde\alpha_i^{\ssup t}\big)_{i\in I} \in \X.
\end{aligned}
\end{equation}
Recall that $\mathscr G_t(\alpha_i)\overset{\ssup d}=\mathscr G_t(\alpha_i\star\delta_a)$. Likewise, we also have 
$(\alpha_i\star\delta_a)^{\ssup t}(\mathrm{d}x)\overset{(d)}{=}(\alpha_i^{\ssup t}\star\delta_a)(\mathrm{d}x)$.
Recall that $\Mcal_1(\X)$ denotes the space of probability measures on $\X$. 
For any $\vartheta\in \Mcal_1(\X)$, then \eqref{alpha after time t} further defines a transition  kernel
\begin{align}
\label{pi-function}
\Pi_t(\vartheta,\d\xi^\prime)= \int_{\X} \pi_t(\xi,\d\xi^\prime) \vartheta(\d\xi) \qquad \mbox{where}\quad \pi_t(\xi,\mathrm{d}\xi^{\prime})=\P\big[\xi^{\ssup t}\in\mathrm{d}\xi^\prime |\xi\big]\in \Mcal_1(\X).
\end{align}
With the above definition, let us record two useful facts. 
\begin{lemma}\label{lemma-m}
The set
\begin{equation}\label{mathcal-K}
\mathfrak {m}_\gamma=\big\{\vartheta\in \mathcal{M}_1(\X):\Pi_t\,\vartheta=\vartheta\text{ for all } t>0\big\}
\end{equation}
of fixed points of $\Pi_t$ is a non-empty, compact subset of $\Mcal_1(\X)$.
\end{lemma}
\begin{proof}
Note that ${\mathfrak m_\gamma}\neq \emptyset$, because $\delta_{\widetilde 0}\in{\mathfrak m_\gamma}$. Moreover, by the definition of the metric $\mathbf D$ on $\X$ and by the resulting convergence criterion determined by Theorem \ref{thm-compact}, the map $\X\ni\xi\mapsto\pi_t(\xi,\cdot)$ is continuous. This property, together with the compactness of $\X$ (and therefore also that of $\Mcal_1(\X)$), we have  that $\Mcal_1(\X)\ni \vartheta\mapsto \Pi_t(\vartheta,\cdot)$ is continuous too for any $t>0$. It follows that ${\mathfrak m_\gamma}$ is a closed subset of the compact metric space $\mathcal M_1(\X)$, implying the compactness of ${\mathfrak m_\gamma}$. 
\end{proof}

For the next lemma we let $\mathscr L(\X)$ denote the space of all Lipschitz functions $f : \X\to \R$ with Lipschitz constant at most $1$ and $f(\widetilde 0)=0$.

 \begin{lemma}\label{LLN}
 For $\vartheta,\vartheta^\prime\in\mathcal M_1(\X)$ let $\mathscr{W}(\vartheta,\vartheta^\prime)=\sup_{f\in\mathscr L(\X)} \big| \int_{\X} f(\xi) \vartheta(\d\xi) - \int_{\X} f(\xi) \vartheta^\prime(\d\xi)\big|$ be the Wasserstein metric on $\Mcal_1(\X)$. If $\bQ_t:=\hP_t[\omega_t\in \cdot]\in \Mcal_1(\R^d)$ (so that $\widetilde\bQ_t\in \X$) and
 $\nu_T:=\frac 1 T\int_0^T \delta_{\widetilde\bQ_t}\d t\in \Mcal_1(\X)$,  then for any $s\geq 0$, $\P$-almost surely, $\mathscr {W}(\nu_T,\Pi_s\nu_T)\to 0$ as $T\to\infty$.
 \end{lemma}

\begin{proof}
Let $f : \X\to \R$ be any Lipschitz function vanishing at $\widetilde 0\in \X$ and having Lipschitz constant at most $1$. Then 
for any $t \in [0,s)$, 
	$M_{n,t}= \sum_{k=0}^n \big(f(\widetilde\bQ_{t+ks+s})-\E\big[ f(\widetilde\bQ_{t+ks+s}) |\mathcal{F}_{t+ks}\big]\big)$ 
	is an $(\mathcal F_{t+ (n+1)s})_{n\in\N_0}$ martingale (here $\mathcal F_t$ is the $\sigma$-algebra generated by the noise ${\dot B}$ up to time $t$). Then by the Burkholder-Davis-Gundy inequality, for some constant $C>0$, $\E[ M_{n,t}^4] \leq C (n+1)^2$, and thus, by Jensen's inequality 
	$$
	\E\bigg[\bigg(\int_0^{sn} f(\widetilde\bQ_{t+s})-\E\big[ f(\widetilde\bQ_{t+s}) |\mathcal{F}_t\big] \,\,\mathrm{d}t\bigg)^4\bigg] = \E\bigg[\bigg(\int_0^s M_{n-1,t} \, \d t\bigg)^4\bigg]\leq C n^2.
	$$
	for some $C=C(s,f)>0$. Hence $\P\big\{\big|\int_0^{sn} f(\widetilde\bQ_{t+s})-\E\big[ f(\widetilde\bQ_{t+s}) |\mathcal{F}_t\big]\mathrm{d}t\big|\geq (sn)^{4/5}\big\}$ is summable and it follows from Borel-Cantelli Lemma that
	\begin{align}\label{lipschitz concentration}
	\limsup_{T\to\infty}T^{-4/5}\bigg|\int_0^{sn} f(\widetilde\bQ_{t+s})-\E\big[ f(\widetilde\bQ_{t+s}) |\mathcal{F}_t\big] \,\,\mathrm{d}t\bigg|\leq C(s,f)
	\end{align}
	almost surely. Since $f$ has Lipschitz constant at most $1$, we also have (choosing $n=\lfloor T/s\rfloor$) $\big|\int_{sn}^Tf(\widetilde\bQ_{t+s})-\E[f(\widetilde\bQ_{t+s})|\mathcal F_t]\d t\big|\leq 2s$. 
	Now for any $s\geq 0$, let
		$\nu_T^{\ssup s}=\frac{1}{T}\int_0^T\delta_{\widetilde\bQ_{t+s}}\mathrm{d}t$, so that $\nu_T=\nu_T^{\ssup 0}$. Then (with $\Pi_s$ as in \eqref{pi-function}), 
		the above display implies that, for $T>0$ sufficiently large, we have that almost surely 
		$$
		\bigg|\int_{\X} f(\xi) \nu_T^\ssup{s}(\d\xi) - \int_{\X} f(\xi) \Pi_s\nu_T(\d\xi)\bigg| \leq C(s,f)T^{-1/5}. 
		$$
On the other hand,
		$$
		\bigg|\int_{\X} f(\xi) \nu_T(\d\xi) - \int_{\X} f(\xi) \nu_T^\ssup{s}(\d\xi)\bigg|=\bigg|\frac{1}{T}\int_0^T f(\widetilde\bQ_t)\mathrm{d}t-\frac{1}{T}\int_s^{T+s} f (\widetilde\bQ_t)\mathrm{d}t\bigg|\leq \frac{2s}{T}.
		$$
		Combining the last two displays, and using triangle inequality, we obtain that almost surely
	\begin{equation}\label{prop for LLN}
	\limsup_{T\to\infty}T^{1/5}\bigg| \int_{\X} f(\xi) \nu_T(\d\xi) - \int_{\X} f(\xi) \Pi_s\nu_T(\d\xi)\bigg|\leq C(s,f).
		\end{equation}	
		By the definition of the metric $\mathbf D$ on $\X$, for any $f\in \mathscr L(\X)$, 
				$\sup_{\xi\in \X}| f (\xi)|\leq\sup_{\xi\in\X}\mathbf{D}(\xi,\widetilde 0)\leq 2$ and 
		thus, $\mathscr L(\X)$ forms an equicontinuous family which is closed in the uniform norm. By Ascoli's theorem, this space is then compact and is also separable.
		If $\Psi_T(f):=\frac{1}{T}\int_0^T f(\widetilde\bQ_{t})-\E\big[ f(\widetilde\bQ_{t+s}) |\mathcal{F}_t\big] \,\,\mathrm{d}t$, then given any $f,g\in\mathscr L(\X))$ with $\| f- g\|_{\infty}<\delta$, we have $|\Psi_T(f)- \Psi_T( g)|<  2\delta$. 
		Thus $(\Psi_T(\cdot))_{T\geq 0}$ is equicontinuous on the compact metric space $\mathscr L(\X)$, and by \eqref{prop for LLN}
		this family converges pointwise to $0$ on a dense subset $(f_n)_n$. Thus, by
		applying Arzela-Ascoli theorem once more, we obtain that this convergence is uniform, which, together with \eqref{prop for LLN} implies the corollary.
		\end{proof}

\subsection{Malliavin calculus and free energy derivatives.}

We recall some rudimentary facts from Malliavin calculus (see e.g. \cite{N06}) and its consequences for our Hamiltonian $\mathscr H_T$. 
Let   $(\mathcal E, \mathcal F, \P)$ be a       complete probability space  carrying a centered Gaussian process $\{{\dot B}(h)\}_{h\in L^2([0,T]\otimes \R^d)}$ with covariance structure 
$\E[{\dot B}(h) {\dot B}(g)]=\langle h,g\rangle_{L^2([0,T]\otimes \R^d)}$.

For any square integrable random variable $F$ on $(\mathcal E, \mathcal F, \P)$, 
the Malliavin
derivative $DF$ is (when it exists) a random element of $L^2([0,T]\otimes \R^d)$, that can be viewed as a space-time indexed stochastic
process $DF=(D_{t,x}F)_{t,x}$. In a particular set up, if  
\begin{equation}\label{form}
F=f({\dot B}(h_1),...,{\dot B}(h_n))
\end{equation}
for a $C^\infty$-function $f:\R^n\rightarrow \R$, then, the Malliavin derivative is defined as 
\begin{equation}\label{Malliavin}
DF=\sum_{i=1}^n\partial_if({\dot B}(h_1),...,{\dot B}(h_n))h_i. 
\end{equation} 
The iterated derivative $D^{\ssup k}F$ is a random element in the tensor product $L^2([0,T]\otimes \R^d)\otimes\dots\otimes L^2([0,T]\otimes \R^d)$. 

For any $p\geq 1$ and any positive integer $k\geq 1$, note that 
$$
\Vert F\Vert_{k,p}=\bigg[\E[|F|^p]+\sum_{i=1}^k\E[\Vert D^\ssup{i}F\Vert_{(L^2)^{\otimes j}}^p]\bigg]^{1/p}
$$ 
defines a semi-norm, and as the domain of the Malliavin derivative $D$ in $L^p(\P)$ is denoted by $\mathbb D^{\ssup{1,p}}$ in the sense that 
$\mathbb D^{\ssup{1,p}}$ is the closure of the class of random variables of the form \eqref{form} with respect to the norm $\Vert \cdot\Vert_{1,p}$. Similarly,  $\mathbb D^{\ssup{k,p}}$ will stand for the
completion of the family of smooth random variables with respect to the norm $\Vert\cdot\Vert_{k,p}$.

\smallskip

With our space-time white noise ${\dot B}$ in our particular set up, note that for a fixed Brownian path $\omega$, the object 
$\mathscr H_T(\omega, {\dot B})=\int_0^T\int_{\R^d}\kappa(y-\omega_s){\dot B}(s,y)\d y\d s$
can be reinterpreted as 
$$
{\dot B}(h)=\int_0^T\int_{\R^d}h(s,y){\dot B}(s,y)\d y\d s, \quad\text{with } \qquad h(s,y)=\kappa(y-\omega_s) \in L^2([0,T]\otimes \R^d).
$$
In particular, if $n=1$ and $f(x)=x$, then the definition \eqref{Malliavin} dictates that $D{\dot B}(h)=h$ and we have the following implications pertinent to the Malliavin derivative of $\mathscr H_T(\omega)$ and the free energy 
\begin{align}
\label{definition of the functional f}
f_T({\dot B})= f_T(\gamma, {\dot B})=\frac 1 T\log Z_T= \frac{1}{T}\log\bE_0\big[\e^{\gamma\int_0^T\int_{\R^d}\kappa(y-\omega_s){\dot B}(s,y)\d y\d s}\big].
\end{align}
Note that, 
\begin{equation}\label{fL2}
T^2\E [f_T({\dot B})^2]\leq \bE_0\E[\e^{\gamma \mathscr H_{T}(\omega,{\dot B})}]+\bE_0\E[\e^{-\gamma\mathscr H_{T}(\omega,{\dot B})}]=2\e^{\frac{\gamma^2}{2}TV(0)}<\infty,
\end{equation}
where we used that for any $x>0$, $(\log(x))^2\leq x+x^{-1}$ and Jensen's inequality. Hence,   for any $T,\gamma\geq 0$, $f_T({\dot B}) \in L^2(\mathbf P)$.
Moreover, note that  (with $\widehat{\bE}_T$ denoting expectation w.r.t. $\hP_T$, cf. \eqref{notation}), 
\begin{equation}\label{fprime}
f^\prime_T(\gamma,{\dot B}):= \frac{\partial}{\partial\gamma} f_T(\gamma,{\dot B})= \frac 1 T \int_0^T\int_{\R^d} \bET[\kappa(y-\omega_s)] {\dot B}(s,y) \d y \d s
\end{equation}
so that, by It\^o isometry and Jensen's inequality, 
\begin{equation}\label{fprimeL2}
T^2\E\big[(f^\prime_T(\gamma,{\dot B}))^2\big] \leq \int_0^T\int_{\R^d} \E\big[\bET(\kappa^2(y-\omega_s))\big] \d y\d s=T(\kappa\star\kappa)(0)<\infty. 
\end{equation}

Recall \eqref{Mprod}; we will also need the following expressions for the Malliavin derivatives of $\mathscr H_T$ and that of the free energy $f_T$: 

\begin{lemma}
	For any $t>0$ and $x\in \R^d$, 
	\begin{itemize}
		\item $D_{t,x}[\mathscr H_T(\omega)]=\kappa(x-\omega_t)$ and 
		\begin{align}
		\label{first Malliavin of free energy}
		D_{t,x} [f_T({\dot B})]
		=\frac{\gamma}{T}\bET[\kappa(x-\omega_t)].
		\end{align}
		\item Moreover, the second Malliavin derivative of the free energy is given by 
		\begin{align}
		\label{second Malliavin of free energy}
		\begin{split}
		D_{t,x}^{\ssup 2}[f_T({\dot B})]
		=\frac{\gamma^2}{T}\bigg(\bET[\kappa(x-\omega_t)^2]-\bET^\otimes[\kappa(x-\omega_t)\kappa(x-\omega^\prime_t)]\bigg)
		\end{split}
		\end{align}
		where $\bET^\otimes$ denotes expectation w.r.t. the product GMC measure $\hP^\otimes_T$ defined w.r.t. two independent Brownian paths. 
		\item For any smooth random variable $F$ of the form \eqref{form}, we have
		\begin{equation}\label{byparts}
		\E[F{\dot B}(h)]=\E\bigg[\int_0^T\int_{\R^d}hD_{t,x}F\d x\d t\bigg].
		\end{equation}
	\end{itemize}
\end{lemma}
\begin{proof}
	Note that the first assertion is a consequence of the definition of Malliavin derivative, while \eqref{first Malliavin of free energy} follows from the chain rule and the fact 
	$$
	D_{t,x} [f_T({\dot B})]=\frac{\gamma}{T}\frac{\bE_0[\kappa(x-\omega_t)\e^{\gamma\mathscr H_{T}(\omega)}]}{Z_T}.
	$$
	Also, the assertion \eqref{second Malliavin of free energy}
	follows from 
	$$
	\begin{aligned}
	&D_{t,x}^{\ssup 2}[f_T({\dot B})]\\
	&=\frac{\gamma^2}{T}\frac{\bE_0[\kappa(x-\omega_t)^2\e^{\gamma\mathscr H_{T}(\omega)}]}{\bE_0[\e^{\gamma\mathscr H_{T}(\omega)}]}-\frac{\gamma^2}{T}\bE_0\big[\kappa(x-\omega_t)\e^{\gamma\mathscr H_{T}(\omega)}\big]\bigg(\frac{\bE_0[\kappa(x-\omega_t)\e^{\gamma\mathscr H_{T}(\omega)}]}{\big(\bE_0[\e^{\gamma \mathscr H_{T}(\omega)}]\big)^2}\bigg).
	\end{aligned}
	$$
	Finally, \eqref{byparts} is an easy consequence of integration by parts for Malliavin calculus that asserts that on any Hilbert space $H$ and $h\in H$, 
	\begin{align}
	\label{integration by parts formula}
	\E[F{\dot B}(h)]=\E[\langle DF,h\rangle_H].
	\end{align}
\end{proof}

\subsection{Flows on path space and a Poincar\'e inequality}

Using the tools from Malliavin calculus from previous section, the goal of the current section is to prove a Poincar\'e inequality defined w.r.t. a {\it flow} for an infinite dimensional Ornstein-Uhlenbeck (O-U) process.

Let us first recall the definition of an Ornstein-Uhlenbeck operator defined w.r.t. a standard Gaussian measure $\mu$ in finite dimensions $\R^n$. Note that if $f:\R^n \to \R$ is differentiable, then its gradient $\nabla f\colon \R^n \to \R^n$ defines a vector field and the divergence $\delta: \R^n\to \R$ can be thought of as an ``adjoint" for $\nabla$ in the Hilbert space $L^2(\R^n)$, i.e. $\delta$ acts on a vector field $v:\R^n\to \R^n$ via the relation
$\E^\mu[\nabla f \cdot v]=\E^\mu[f \delta v]$. Via this relation we also have, for any $v=(v^{\ssup1},\dots, v^{\ssup n})$ with $v^{\ssup i}:\R^n\to \R$ and $f:\R^n\to \R$ continuously differentiable, by integration by parts 
$\E^\mu[\nabla f\cdot v]= \sum_{i=1}^n \int_{\R^n} \partial_i f(x) v^{\ssup i}(x) \mu(\d x)= \sum_{i=1}^n \int_{\R^n}  f(x) \big(x_iv^{\ssup i}(x)- \partial_i v^{\ssup i}(x)\big) \mu(\d x)$, and consequently, 
$\delta v=\sum_{i=1}^n (x_i v^{\ssup i}- \partial_i v^{\ssup i})$. In particular, the latter identity implies for $v= f\nabla g: \R^n \to \R^n$ and any sufficiently smooth $f, g:\R^n\to \R$, that 
\begin{equation}\label{identity}
\begin{aligned}
\delta(f\nabla g)&= \sum_{i=1}^n \bigg[x_i f(x) \frac{\partial g}{\partial x_i} - \bigg(\frac {\partial f}{\partial x_i}\frac{\partial g}{\partial x_i} + f(x) \frac{\partial^2 g}{\partial x_i^2} \bigg)\bigg] 
\\
&= - \nabla f\cdot\nabla g - f(\Delta g- x\cdot \nabla g).
\end{aligned}
\end{equation}
Now we can define the Ornstein-Uhlenbeck operator $\mathcal L$ as 
\begin{equation}\label{OU1}
\mathcal L= - \delta \circ \nabla,
\end{equation}
and given \eqref{identity} (for the particular choice $f=1$), the above definition reduces to 
\begin{equation}\label{OU2}
\mathcal L g= -\delta (\nabla g)= \Delta g- x \cdot\nabla g.
\end{equation}

The above finite dimensional setup can be translated to the abstract Gaussian space $(\mathcal E, \mathcal F, \P)$ too by replacing the gradient by the Malliavin derivative $D$ defined before, while the divergence 
$\delta$ acts as an adjoint of $D$. In other words, for any $F\in H= \mathbb D^{\ssup{1,2}}$ and $u$ in the domain of $\delta$, we have 
$$
\E[\langle DF,u\rangle_H]=\E[F\delta(u)]
$$
and the Ornstein-Uhlenbeck operator $\mathscr L$ for $\mathcal E$ then is defined by
\begin{equation}\label{OU3}
\mathscr L=-\delta\circ D.
\end{equation} 
Recall the integration by parts formula for Malliavin calculus \eqref{integration by parts formula}. We apply that formula to the product of the two random variables $F,G$ of the form \eqref{form}. Then,
$$
\E[G\langle DF,h\rangle_H]=-\E[F\langle DG,h\rangle_H]+\E[FG{\dot B}(h)].
$$
If $u$, which is in the domain of $\delta$ has the form $u=\sum_{j=1}^nF_jh_j$, then by the last display,
$$
\E[\langle DF,u\rangle_H]=\sum_{j=1}^n\E[F_j\langle DF,h_j\rangle_H]=\sum_{j=1}^n-\E[F\langle DF_j,h\rangle_H]+\E[FF_j{\dot B}(h_j)]
$$
and we conclude that $\delta(u)=\sum_{j=1}^nF_j{\dot B}(h_j)-\langle DF_j,h_j\rangle_H$. For the Ornstein-Uhlenbeck operator $\mathscr L$ this means $\mathscr L u=-\delta(Du)=\sum_{j=1}^n \langle D^{\ssup{2}}F_j,h_j\rangle_H-DF_j{\dot B}(h_j)$. Applying this theory in our setting to the functional $f_T({\dot B})$, the Ornstein-Uhlenbeck operator has the form
\begin{equation}\label{OU4}
 (\mathscr L f_T)({\dot B})=\int_0^T\int_{\R^d}D_{t,x}^{\ssup 2}f_T({\dot B})\d x\d t- {\dot B}\big( Df_T({\dot B})\big).
\end{equation}

Let $\eta$ be space-time white noise which is an {\it independent} copy of ${\dot B}$. That is, $\{\eta(f)\}_{f\in L^2([0,T]\otimes \R^d)}$ is
a centered Gaussian process with covariance $\E[\eta(f_1) \eta(f_2)]=\langle f_1, f_2 \rangle_{L^2([0,T]\otimes \R^d)}$.

We define the {\it Ornstein Uhlenbeck flow} of ${\dot B}$ at time $t\geq 0$ by 
\begin{equation}\label{flow}
\boldsymbol{\dot B}_t(s,x)=\e^{-t}{\dot B}(s,x)+\e^{-t}\eta\big((\e^{2t}-1)^{-1}s, x\big)\text{ if }t>0, \qquad \boldsymbol{\dot B}_0={\dot B}.
\end{equation}
Recall that 
$\lambda \theta^{d/2}{\dot B}(\lambda^2 s, \theta x)$ has the same law 
as that of ${\dot B}(s,x)$ for any $\lambda, \theta>0$. Since $\eta$ is an independent copy of ${\dot B}$, it follows that
for any fixed $t>0$, $\{\boldsymbol{\dot B}_t(f)\}_f$ is also a centered Gaussian process with the same covariance structure $\E[\boldsymbol{{\dot B}}_t(f_1) \boldsymbol{{\dot B}}_t(f_2)]=\langle f_1, f_2 \rangle_{L^2([0,T]\otimes \R^d)}$. 

Therefore, we can define 
\begin{equation}
\begin{aligned}
&\mathscr H_{T}(\omega,\boldsymbol{{\dot B}}_t)= \int_0^T\int_{\R^d} \kappa(y-\omega_s) \boldsymbol{{\dot B}}_t(s,y) \d y \d s, \\
& {\hP}_T(\boldsymbol{{\dot B}}_t)=\frac 1 {Z_T(\boldsymbol{{\dot B}}_t)} \exp\big\{\gamma \mathscr H_{T}(\omega,\boldsymbol{{\dot B}}_t)\big\} \bP_0(\d \omega), \quad
Z_T(\boldsymbol{{\dot B}}_t)=\bE_0\big[\exp\big\{\gamma \mathscr H_{T}(\omega,\boldsymbol{{\dot B}}_t)\big\} \big]. 
\end{aligned}
\end{equation}
For expectation $\bE^{{\hP}_T(\boldsymbol{{\dot B}}_t)}$ with respect to the probability measure ${\hP}_T(\boldsymbol{{\dot B}}_t)$ we write $\bETt$. And if $\omega,\omega'$ are two independent Brownian motions, we write
\begin{align*}
\hP_T^\otimes(\Bt)&=\frac{1}{Z_T(\boldsymbol{{\dot B}}_t)^2}\exp\big\{\gamma(\mathscr H_T(\omega,\boldsymbol{{\dot B}}_t)+\mathscr H_T(\omega^\prime,\boldsymbol{{\dot B}}_t))\big\}\bP_0^\otimes(\d \omega,\d \omega^\prime)
\end{align*}  
and for expectation with respect to the probability measure  $\hP_T^\otimes(\Bt)$ we write $\bET^{\ssup{\Bt}^\otimes}$.

We also need 
\begin{lemma}\label{claim2}
	For any $T,\gamma>0$, 
	\begin{align}
	\label{OU operator applied to f}
	\mathscr  L f_T(\Bt)=\gamma^2V(0)-\frac{\gamma^2}{T}\int_0^T\bET^{\ssup{\Bt}^\otimes}[V(\omega_s-\omega_s^\prime)]\d s-\gamma f_{T}^\prime(\gamma,\Bt) \in L^2(\P). 
	\end{align}
	for $f_T$ as in \ref{definition of the functional f} and $f_{T}^\prime(\gamma,\Bt) =\frac{\partial}{\partial\gamma}f_T(\gamma,\Bt)$.
\end{lemma}

\begin{proof} Indeed, recall \eqref{first Malliavin of free energy} and \eqref{second Malliavin of free energy} for the first two Malliavin derivatives of $f_T$. 
	Then,
	\begin{align}
	\label{skalar with Malliavin derivative}
	\Bt\big( D f_T(\Bt) \big)=\int_0^T\int_{\R^d}\frac{\gamma}{T}\bETt[\kappa(y-\omega_s)]\Bt(s,y)\d y\d s \stackrel{\eqref{fprime}}=\gamma f_{T}^\prime(\gamma,\Bt)
	\end{align}
	and with the Ornstein-Uhlenbeck generator, see \eqref{OU4}, we have \eqref{OU operator applied to f}. Moreover, since $0\leq V(\cdot)\leq V(0)$, 
	by using \eqref{fprimeL2}, we have 
	$\E[\mathscr  L f_T(\Bt)^2]=\E[\mathscr  L f_T({\dot B})^2] \leq C(\gamma^4 V(0)^2+ \|f^\prime(\gamma,\cdot)\|_{L^2(\P)}^2)<\infty$.
\end{proof}

The following Poincar\'e inequality will be quite useful in the context of proving Theorem \ref{localization via overlaps}. 
\begin{lemma}
	\label{Lemma 4.3}
	Let $f_T(\boldsymbol{{\dot B}}_t)$ be the functional defined in \eqref{definition of the functional f} w.r.t. the flow $\boldsymbol{{\dot B}}_t$ defined in \eqref{flow}. Then, for any $T$ and $\gamma$, 
	$$
	\mathrm{Var}^{\P}\bigg(\frac{1}{t}\int_0^t\mathscr L f_T(\boldsymbol{{\dot B}}_r)\d r\bigg)\leq \frac{2}{t}\E\big[\Vert Df_T({\dot B})\Vert^2_{L^2([0,T]\otimes \R^d)}\big].
	$$
\end{lemma}
\begin{proof}
	Let $\mathrm{Var}$ and $\mathrm{Cov}$ stand for variance and covariance w.r.t. $\P$, while $(\mathscr P_t)_{t\geq 0}$ stands for the Ornstein-Uhlenbeck semigroup. That is, for $t\geq 0$, and a test function $g\in L^2(\P)$ defined on the path space of the white noise so that $\mathscr Lg\in L^2(\P)$, 
	$$
	(\mathscr P_t g)({\dot B})=\E\big[g(\Bt)\big|{\dot B}\big].
	$$
	Then, for any $0\leq s\leq t$,
	\begin{align*}
	\mathrm{Cov}(g(\Bs),g(\Bt))&=\mathrm{Cov}(g(\Bs),\E[g(\Bt)|\Bs])\\
	&=\mathrm{Cov}(g(\Bs),\mathscr P_{t-s}g(\Bs))=\mathrm{Cov}(g({\dot B}),\mathscr P_{t-s}g({\dot B})).
	\end{align*}
	Now let $\{\psi_j\}_{j\geq 0}$ be an orthonormal basis of $L^2(\P)$ consisting of eigenfunctions of $\mathscr L$, with $\psi_0\equiv 1$, $\mathscr L\psi_0=\lambda_0\psi_0=0$ and $\mathscr L\psi_j=-\lambda_j\psi_j$ with $\lambda_j>0$ for $j\geq 1$. Then, for $g=\sum_{j\geq 0}a_j\psi_j\in L^2(\P)$, we have
	$$
	\mathscr Lg=-\sum_{j\geq 1}\lambda_ja_j\psi_j,\quad \mathscr P_t\mathscr L g=-\sum_{j\geq 1}\lambda_ja_j\e^{-\lambda_jt}\psi_j.
	$$
	Further, if $g_1=\sum_{j\geq 0}a_j\psi_j$, $g_2=\sum_{j\geq 0}b_j\psi_j\in L^2(\P)$, then
	$$
	\mathrm{Cov}(g_1({\dot B}),g_2({\dot B}))=\sum_{j\geq 1}a_jb_j
	$$
	and if in addition $D^{\ssup 2}g_1,D^{\ssup 2}g_2$ exist, then
	\begin{align}
	\label{Dirichlet form}
	-\E[g_1({\dot B})\mathscr Lg_2({\dot B})]=\E[D g_1({\dot B})D g_2({\dot B})].
	\end{align} 
	Hence,
	$
	\mathrm{Cov}(\mathscr Lg(\Bs),\mathscr L g(\Bt))=\mathrm{Cov}(\mathscr Lg({\dot B}_0),\mathscr P_{t-s}\mathscr L g({\dot B}_0))=\sum_{j\geq 1}\lambda_j^2a_j^2\e^{-\lambda_j(t-s)}.
	$
	Again, if $g=\sum_{\geq 0}a_j\psi_j$, then by \eqref{Dirichlet form},
	$$
	\E\Vert Dg({\dot B})\Vert^2=-\E[g({\dot B})\mathscr L g({\dot B})]=\sum_{j\geq 1}\lambda_ja_j^2.
	$$
	Thus,
	$$
	\int_0^t\mathrm{Cov}(\mathscr L g(\Br),\mathscr Lg(\Bt))\d r=\sum_{j\geq 1}\int_0^t\lambda_j^2a_j^2\e^{-\lambda_j(t-r)}\d r=\sum_{j\geq 1}\lambda_ja_j^2(1-\e^{-\lambda_jt})\leq \E\Vert Dg({\dot B})\Vert^2
	$$
	and finally
	\begin{align*}
	\mathrm{Var}\bigg(\frac{1}{t}\int_0^t\mathscr Lg(\Br)\d r\bigg)=\frac{1}{t^2}\int_0^t\int_0^t\mathrm{Cov}(\mathscr Lg(\Br),\mathscr Lg(\Bs))\d s\d r\\
	=\frac{2}{t^2}\int_0^t\int_0^r\mathrm{Cov}(\mathscr Lg(\Br),\mathscr Lg(\Bs))\d s\d r\leq \frac{2}{t}\E\Vert Dg({\dot B})\Vert^2.
	\end{align*}
	We now choose $g({\dot B})=f_T(\gamma,{\dot B})$ and since $f_T, \mathscr L f_T\in L^2(\P)$, we apply the above bound. 
\end{proof}

\begin{cor}\label{Proposition 4.0}
	For any $\eps>0$ and $\gamma>0$, 
	$$
	\P\bigg[\frac 1 t\int_0^t (\mathscr L f_T)(\Br) \d r > \frac{\gamma \eps}2\bigg] \leq \frac{8V(0)}{tT\eps^2}.
	$$
\end{cor}

\begin{proof}
	Recall that $\{\psi_j\}_{j\geq 0}$ is an orthonormal basis of $L^2(\P)$ consisting of eigenfunctions of $\mathscr L$, with $\psi_0\equiv 1$. Since $\mathscr Lf_T(\cdot)= \sum_{j\geq 1} a_j \lambda_j \psi_j(\cdot)$ and 
	$\psi_j \bot 1$ for all $j\geq 1$, we have $\E[\mathscr L f_T(\Bt)]=0$. Then by Lemma \ref{Lemma 4.3} and by Jensen's inequality, 
	\begin{align*}
	\mathrm{Var}\bigg[\frac 1 t\int_0^t (\mathscr L f_T)(\Br) \d r\bigg]
	&\leq \frac{2}{t}\E\bigg[\frac{\gamma^2}{T^2}\int_0^T\int_{\R^d} \bET[\kappa^2(y-\omega_s)]\d y\d s\bigg]= \frac{2\gamma^2V(0)}{tT}.
	\end{align*}
	Therefore, the claim follows by Chebyshev's inequality.
\end{proof}

\section{Proofs of main results: Theorem \ref{Theorem 2 new} - Theorem \ref{localization via overlaps}}\label{sec-proofs}


\subsection{Proof of Theorem \ref{Theorem 2 new}.} Theorem \ref{Theorem 2 new} will be proved in four steps.

\noindent{\bf Step 1:} For any $x\in\R^d$ and $A\in\mathcal G$ we set (recall the notation from \eqref{notation}-\eqref{Mprod})
$$\mathscr Z_T^\ssup{x}(A):=\bE_x\big[\1_A\,\,\e^{\gamma\mathscr H_T(\omega)-\frac{\gamma^2}{2}TV(0)}\big].$$
\begin{lemma}
	\label{Rewrite of the partition function}
	We set $M_T^\ssup{1}=\gamma\int_0^T\int_{\R^d}\bEt\big[\kappa(y-\omega_t)\big]{\dot B}(t,y) \d y  \,\d t$. Let $\bQ_{t}:= \hP_t[\omega_t\in \cdot] \in \Mcal_1(\R^d)$ so that 
$\widetilde{\bQ}_t \in \X$. Let $F_\gamma$ and $\overline{\mathscr G}_t$ be the functionals defined in \eqref{functional Phi} and \eqref{scr F}, respectively.  Then 
\begin{equation}\label{rewrite-partition}
	\begin{aligned}
	\frac 1T &\log \mathscr Z_T=\frac 1 TM_T^\ssup{1}-\frac 1 T  \int_0^TF_\gamma(\widetilde{\bQ}_t) \d t \qquad \qquad\text{and}\\
	\frac{1}{T}&\log(\overline{\mathscr G}_T(\xi))
	=\frac{1}{T}M_T^\ssup{2}-\frac{\gamma^2}{2T}\int_0^T\d t\bigg( \sum_{\widetilde{\alpha}_1,\widetilde{\alpha}_2\in\xi}\int_{\R^{2d}}V(x_1-x_2)\\
	&\qquad\qquad\qquad \qquad\qquad\qquad\qquad\qquad\times \prod_{j=1}^2\frac{1}{\overline{\mathscr G}_t(\xi)}\int_{\R^d}\alpha_j(\d z_j)\mathscr Z_t^\ssup{z_j}(\omega_t^\ssup{j}\in\d x_j)\bigg),
	\end{aligned}
	\end{equation}
	where $M_T^\ssup{2}$ is a mean zero martingale defined below in \eqref{definition M^2}.
\end{lemma}	
\begin{proof}
Writing $Z_T= \bE_0[\e^{\gamma \mathscr H_T}]$ and applying It\^o's formula to $\log Z_T$ we have 
\begin{equation}\label{ub0}
\d \log Z_T= \frac 1 {Z_T} \d Z_T - \frac 1 {2Z_T^2} \d\langle Z_T\rangle
\end{equation}
Again by It\^o's formula, 
\begin{equation}\label{ub1}
\d Z_t= \bE_0\bigg[\gamma \int_{\R^d}  \e^{\gamma \mathscr H_t(\omega)} \kappa(y-\omega_t) {\dot B}(t,y) \d y\bigg] \d t + \bE_0\bigg[\frac {\gamma^2}2\int_{\R^d} 
\e^{\gamma \mathscr H_t(\omega)} \kappa(y-\omega_t)^2 \d y\bigg]\d t.
\end{equation}
The quadratic variation of $Z_t$ is also given by 
       \begin{equation}\label{ub2}
	\begin{aligned}
	\mathrm{d}\langle Z_t\rangle&=\mathrm{d}\bigg\langle  \bE_0\bigg[\gamma \int_{\R^d}\, \e^{\gamma \mathscr H_t(\omega)}\kappa(y-\omega_t){\dot B}(t,y) \d y \bigg]\bigg{\rangle}\\
	&=\gamma^2\bE_0^{\otimes }\left[\int_{\R^d}\, \e^{\gamma(\mathscr H_t(\omega)+\mathscr H_t(\omega^{\prime}))}\kappa(y-\omega_t)\kappa(y-\omega_t^{\prime})\d y\right] \d t\\
	&=\gamma^2\bE_0^{\otimes }\left[V(\omega_t-\omega_t^{\prime})\,\,\e^{\gamma (\mathscr H_t(\omega)+\mathscr H_t(\omega^{\prime}))}\right] \d t
	\end{aligned}
	\end{equation}
	where $\omega^{\prime}$ is another Brownian motion independent of $\omega$. 
	Combining \eqref{ub0}-\eqref{ub2} yields 
	\begin{align*}
	\mathrm{d}\log Z_t&=\gamma \bEt\bigg[\int_{\R^d}\kappa(y-\omega_t){\dot B}(t,y) \d y\bigg]\mathrm{d}t+\frac{\gamma^2}{2}\bEt\bigg[\int_{\R^d}\kappa(y-\omega_t)^2\mathrm{d}y\bigg]\mathrm{d}t \\
	&\qquad\qquad-\frac{\gamma^2}{2}\bEt^{\otimes }\left[V(\omega_t-\omega_t^{\prime})\right]\mathrm{d}t
	\end{align*}
	where we recall from \eqref{Mprod} that $\bEt$ denotes expectation w.r.t. the GMC measure $\hP_t$, while $\bEt^{\otimes }$ stands for the same w.r.t. the product GMC measure $\hP_t^\otimes$. 
	Since $\int_{\R^d} \kappa(y-\omega_t)^2 \d y=\int_{\R^d} \kappa^2(y)\d y= V(0)$, the display above now yields
	\begin{align*}
{\log Z_T}&= \gamma\int_0^T \bEt\left[\int_{\R^d}\kappa(y-\omega_t){\dot B}(t,y)\d y\right]\mathrm{d}t+\frac{\gamma^2 T V(0)}2
-\frac{\gamma^2}{2}\int_0^T\bEt^{\otimes }\left[V(\omega_t-\omega_t^{\prime})\right]\mathrm{d}t,
\end{align*}
where $\omega,\omega^\prime$ are independent Brownian motions. Consequently,
$$
\frac{1}{T}\log\mathscr Z_T=\frac{1}{T}\log Z_T-\frac{\gamma^2}{2}V(0)=\frac{M_T^{\ssup 1}}{T}-\frac{\gamma^2}{2T}\int_0^T\bEt^{\otimes }\left[V(\omega_t-\omega_t^{\prime})\right]\mathrm{d}t.
$$
From the above display the first identity in \eqref{rewrite-partition} follows. Repeating the It\^ o computation for $\log(\overline{\mathscr G}_T(\xi))$ also proves the second identity in \eqref{rewrite-partition} with
\begin{align}\label{definition M^2}
M_T^\ssup{2}=\gamma\int_0^T\d t\int_{\R^d}\d y\,\,\frac{1}{\overline{\mathscr G}_T(\xi)}\sum_{\widetilde{\alpha}\in\xi}\int_{\R^d}\alpha(\d z)\bE_z\bigg[\kappa(y-\omega_s)\e^{\gamma\mathscr H_t(\omega)-\frac{\gamma^2}{2}TV(0)}\bigg]{\dot B}(s,y).
\end{align} 	
\end{proof} 

A similar computation as Lemma \ref{Rewrite of the partition function} also provides 

\begin{lemma}\label{lemma-CC}
For any $\gamma>0$, $\delta>0$ and as $T\to\infty$, we have $\log Z_T- \E[\log Z_T] = o(T)$.
\end{lemma}
\begin{proof}
For any $s\in [0,T]$ we set 
$$
X_T=\log Z_{T}-\E[\log Z_{T}], \qquad X_{T,s}=\E\big[\log Z_{T}-\E[\log Z_{T}]|\mathcal F_s\big ], 
$$
where $\mathcal F_s$ is the $\sigma$-algebra generated by the noise ${\dot B}$ up to time $s$. 
We note that $(X_{T,s})_{s\in[0,T]}$ is a martingale and also that
\begin{align*}
X_{T,s}=\E\bigg[\gamma\int_0^T&\int_{\R^d} \bEt[\kappa(y-\omega_t)]{\dot B}(t,y)\d y\mathrm{d}t\\
&+\frac{\gamma^2}{2}\int_0^T\E\bigg[\bEt^{\otimes}\left[V(\omega_t-\omega_t^{\prime})\right]\bigg]-\bEt^{\otimes}\left[V(\omega_t-\omega_t^{\prime})\right]\mathrm{d}t\bigg|\mathcal F_s\bigg].
\end{align*}
The quadratic variation of $X_{T,T}$ is given by 
\begin{align*}
\langle X_{T,T}\rangle
&=\bigg\langle\gamma\int_0^T\int_{\R^d} \bEt[\kappa(y-\omega_t)]{\dot B}(t,y)\d y\mathrm{d}t\bigg\rangle=\gamma^2\int_0^T\int_{\R^d} \Big(\bEt[\kappa(y-\omega_t)]\Big)^2\d y\mathrm{d}t
\leq \gamma^2TV(0), 
\end{align*}
which we will estimate now as follows: Indeed, again by the martingale property of $(X_{T,s})_{s\in[0,T]}$, we have that for any $a\in\R$, $\big(\exp\big\{aX_{T,s}-\frac{a^2}{2}\langle X_{T,s}\rangle\big\}\big)_{s\in[0,T]}$ is also an exponential martingale. Therefore by Chebyshev's inequality, for any $a,u>0$,
\begin{align*}
\P(X_{T}>u)\leq \E\big[\e^{aX_T}\big]\e^{-au}\leq \E\big[\e^{aX_{T,T}-\frac{a^2}{2}\langle X_{T,T}\rangle}\big]\e^{\frac{a^2}{2}\gamma^2TV(0)-au}.
\end{align*}
Since $X_{T,0}=0$, minimizing over $a$ yields
$$
\P(X_{T}>u)\leq\exp\bigg\{\min_{a>0}\Big\{\frac{a^2}{2}\gamma^2TV(0)-au\Big\}\bigg\}=\exp\bigg\{-\frac{u^2}{2\gamma^2TV(0)}\bigg\}.
$$
Since the same calculations hold by replacing $X_T$ by $-X_T$, it follows that, for any $u>0$, 
$$
\P(|\log Z_T-\E \log Z_T|> u)\leq 2\exp\bigg(-\frac{u^2}{2\gamma^2TV(0)}\bigg).
$$
In particular, for $\eta \in \big(\frac 12,1\big)$ and for a sequence defined by $T_1=1$, and $T_{n+1}=T_n+T_n^\eta$, so that $T_n=n^{\frac{1}{1-\eta}+o(1)}$ as $n\rightarrow\infty$,
the upper bound in the last display, combined with Borel-Cantelli lemma, implies that 
\begin{equation}\label{ub4}
\lim_{n\to\infty}\frac{\log Z_{T_n}- \E[\log Z_{T_n}]}{T_n}= 0,\quad\P-\mbox{ a.s.}
\end{equation}
 In order to strengthen the latter assertion for $T\to\infty$, we apply  
Lemma \ref{Rewrite of the partition function} which implies that
$\log Z_T=M_T-\frac{1}{2}\langle M_T\rangle+\frac{\gamma^2}{2}TV(0)$
where $M_T=\gamma\int_0^T\int_{\R^d} \bEt[\kappa(y-\omega_t)]{\dot B}(t,y)\d y\mathrm{d}t$ is a continuous martingale satisfying $\frac{\d}{\d T}\langle M_T\rangle\leq\gamma^2V(0)$ for all $T\geq 0$. We now fix a sequence $\eps_n\rightarrow 0$ such that $\eps_n^{-1}=n^{o(1)}$. For $n$ large enough, $\gamma^2V(0)T_n^\eta<\eps_nT_{n+1}$, and by Doob's inequality, we have 
\begin{align*}
&\P\bigg(\sup_{T_n\leq T\leq T_{n+1}}|\log Z_T-\log Z_{T_n}-\E \log Z_T+\E \log Z_{T_n}|>2\eps_nT_{n+1}\bigg) \\
&\leq \P\bigg(\sup_{T_n\leq T\leq T_{n+1}}|M_T-M_{T_n}|>\eps_nT_{n+1}\bigg)
\leq (\eps_nT_{n+1})^{-2}\E\big[\langle M_{T_{n+1}}\rangle-\langle M_{T_n}\rangle\big],
\end{align*}
and as $\E\big[\langle M_{T_{n+1}}\rangle-\langle M_{T_n}\rangle\big]\leq \gamma^2V(0)(T_{n+1}-T_n)$, the right-hand side above defines a summable series if we choose $\eta\in(1/2,1)$ large enough. Together with \eqref{ub4}, Borel-Cantelli lemma then concludes the proof of the lemma.

\end{proof}

\smallskip

\noindent{\bf Step 2:}  We will now prove 
\begin{lemma}\label{main step for variational}
	With $F_\gamma$ defined in \eqref{functional Phi}, define ${\mathcal E}_{F_\gamma}:\mathcal M_1(\X)\to\R$ to be 
\begin{align}\label{functional I_Phi}
{\mathcal E}_{F_\gamma}(\vartheta)= \int_{\X} F_\gamma(\xi) \, \vartheta(\d\xi).
\end{align}
Then ${\mathcal E}_{F_\gamma}$ is continuous on $\Mcal_1(\X)$. Moreover, with $\Pi_t(\cdot,\cdot)$ defined in \eqref{pi-function}, we have, 
	for any $T>0$ and $\xi\in\X$,$\int_0^T {\mathcal E}_{F_\gamma}(\Pi_t \delta_\xi) \, \d t\leq -\,\E[\log\mathscr  Z_T]$.
\end{lemma}
\begin{proof}
Recall that $F_\gamma$ is continuous on $\X$ and $\X$ is a compact metric space. Thus, ${\mathcal E}_{F_\gamma}$ is also continuous on $\Mcal_1(\X)$. 
We will now prove $\int_0^T {\mathcal E}_{F_\gamma}(\Pi_t \delta_\xi) \, \d t\leq -\,\E[\log\mathscr  Z_T]$. 

Note that by the definition of ${\mathcal E}_{F_\gamma}$ and that of $\Pi_t$, we have for any $t$
\begin{align*}
{\mathcal E}_{F_\gamma}(\Pi_t\delta_{\xi})=\int_{\X}F_\gamma(\xi^{\prime})\Pi_t(\delta_{\xi},\mathrm{d}\xi^{\prime})=\int_{\X}F_\gamma(\xi^{\prime})\, \mathbf P\big[\xi^{\ssup t} \in \d\xi^\prime|\xi\big]=\E\big[F_\gamma(\xi^{\ssup t})\big]. 
\end{align*}
On the other hand, $F_\gamma(\xi^{\ssup t})=\frac{\gamma^2}{2}\sum_{\widetilde{\alpha}\in\xi}\int_{\R^d\times\R^d}V(x_1-x_2)\prod_{j=1}^2\alpha^{\ssup t}(\mathrm{d}x_j)$	and so
\begin{align}
\label{identity between I and Phi}
{\mathcal E}_{F_\gamma}(\Pi_t\delta_{\xi})=\E\bigg[\frac{\gamma^2}{2}\sum_{\widetilde{\alpha}\in\xi}\int_{\R^d\times\R^d}V(x_1-x_2)\prod_{j=1}^2\alpha^{\ssup t}(\mathrm{d}x_j)\bigg].
\end{align}
Recall that $\overline{\mathscr G}_T(\xi)=\mathscr G_T(\xi)+\E\big[\mathscr Z_T-\mathscr G_T(\xi)\big]$. We claim that
\begin{align}
\label{inequality between scr I and log}
\int_0^T{\mathcal E}_{F_\gamma}(\Pi_t\delta_{\xi})\d t\leq -\,\E\big[\log(\overline{\mathscr G}_T(\xi))\big].
\end{align}
For proving \eqref{inequality between scr I and log}, we first consider the sum on the right-hand side of \eqref{identity between I and Phi}. Since $V,\alpha,\mathscr Z_t$ and $\mathscr G_t(\xi)$ are nonnegative,
\begin{align*}
\sum_{\widetilde{\alpha}\in\xi}&\int_{\R^d\times\R^d}V(x_1-x_2)\prod_{j=1}^2\alpha^{\ssup t}(\mathrm{d}x_j)\\
&\leq \sum_{\widetilde{\alpha}_1\in\xi}\sum_{\widetilde{\alpha}_2\in\xi}\int_{\R^d\times\R^d}V(x_1-x_2)\prod_{j=1}^2\frac{1}{\overline{\mathscr G}_T(\xi)}\alpha_j(\d z_j)\bE_{z_j}\big[\1{\{\omega_t^{\ssup j}\in \d x_j\}}\e^{\gamma\mathscr H_t(\omega)}\big],
\end{align*}
thus by \eqref{identity between I and Phi},
\begin{align*}
{\mathcal E}_{F_\gamma}(\Pi_t\delta_{\xi})\leq \E\Bigg[\frac{\gamma^2}{2}\sum_{\widetilde{\alpha}_1\in\xi}\sum_{\widetilde{\alpha}_2\in\xi}\int_{\R^d\times\R^d}V(x_1-x_2)
&\prod_{j=1}^2\frac{1}{\overline{\mathscr G}_T(\xi)}\alpha_j(\d z_j)\bE_{z_j}\big[\1{\{\omega_t^{\ssup j}\in \d x_j\}}\e^{\gamma\mathscr H_t(\omega)}\big]\Bigg].
\end{align*}
Claim \eqref{inequality between scr I and log} now immediately follows from Lemma \ref{Rewrite of the partition function}, as $M_T^\ssup{2}$ in Lemma \ref{Rewrite of the partition function} is a martingale that has expectation $0$.

For any $\xi=(\widetilde\alpha_i)_{i\in I}\in \X$, let
$\sigma(\xi)=\sum_{i\in I}\alpha_i(\R^d)$ (which is well-defined on $\X$ because $\alpha(\R^d)=(\alpha\star\delta_x)(\R^d)$ for any $\alpha\Mcal_{\leq 1}$ and $x\in \R^d$) 
and $\sigma(\cdot)\geq 0$, with identity being true if and only if $\xi=\tilde 0 \in \X$. We now use \eqref{inequality between scr I and log} for those $\xi$ with $\sigma(\xi)>0$. Using the concavity of the logarithm, we obtain 
\begin{align}
\begin{split}
\label{using concavity}
\E \big[\log(\overline{\mathscr G}_T(\xi))\big]&=\E\Big[\log\big(\sigma(\xi)\frac{\mathscr G_T(\xi)}{\sigma(\xi)}+(1-\sigma(\xi))\E\mathscr Z_T\big)\Big]\\
&\geq \sigma(\xi)\E\log\Big(\frac{\mathscr G_T(\xi)}{\sigma(\xi)}\Big)+(1-\sigma(\xi))\log\big(\E \mathscr Z_T\big).
\end{split}
\end{align}
As $\int\sigma(\xi)^{-1}\sum_{\widetilde{\alpha}\in\xi}\alpha(\d x)=1$, we can use Jensen's inequality, so that
\begin{align*}
\log\Big(\frac{\mathscr G_T(\xi)}{\sigma(\xi)}\Big)=\log\bigg(\int_{\R^d}\bigg(\frac{\sum_{\widetilde{\alpha}\in\xi}\alpha(\d z)}{\sigma(\xi)}\bigg)\mathscr Z_T[z]\bigg)
\geq \int_{\R^d}\bigg(\frac{\sum_{\widetilde{\alpha}\in\xi}\alpha(\d z)}{\sigma(\xi)}\bigg)\log\mathscr Z_T[z]
\end{align*}
and since $\mathscr Z_T[z]\overset{(d)}{=}\mathscr Z_T$,
$$
\E\log\Big(\frac{\mathscr G_T(\xi)}{\sigma(\xi)}\Big)\geq \int_{\R^d}\bigg(\frac{\sum_{\widetilde{\alpha}\in\xi}\alpha(\d z)}{\sigma(\xi)}\bigg)\E\log \mathscr Z_T=\E\log\mathscr Z_T.
$$
By using Jensen's inequality once more, $\log\E \mathscr Z_T\geq \E\log \mathscr Z_T$, and both lower bounds, together with \eqref{inequality between scr I and log} and \eqref{using concavity}, yield  $\int_0^T {\mathcal E}_{F_\gamma}(\Pi_t \delta_\xi) \, \d t\leq -\,\E[\log \mathscr Z_T]$ for any $\xi\in \X$ with $\sigma(\xi)>0$. The last inequality, when $\sigma(\xi)=0$, follows immediately by Jensen's inequality. Indeed, if $\sigma(\xi)=0$, then ${\mathcal E}_{F_\gamma}(\Pi_t \delta_\xi)=0$ for all $t$ and so $\int_0^T {\mathcal E}_{F_\gamma}(\Pi_t \delta_\xi) \, \d t=- \log \E \mathscr Z_T\leq -\,\E[\log \mathscr Z_T]$. Thus, the inequality in Lemma \ref{main step for variational} holds unconditionally.
\end{proof}

\smallskip

\noindent{\bf Step 3:} Recall \eqref{pi-function} for definition of $\Pi_t$ and its fixed points 
$\mathfrak {m}=\{\vartheta\in \mathcal{M}_1(\X):\Pi_t\,\vartheta=\vartheta\text{ for all } t>0\}$.
Then the inequality in Lemma \ref{main step for variational} dictates that, for any $\vartheta\in {\mathfrak m_\gamma}$, 
$$
\begin{aligned}
-\frac 1 T\E[\log\mathscr Z_T] \geq \frac 1 T\int_{\X} \vartheta(\d\xi)\,\, \int_0^T\d t\,\, {\mathcal E}_{F_\gamma}(\Pi_t \delta_\xi)  = \frac 1 T \int_0^T \d t {\mathcal E}_{F_\gamma}(\Pi_t\vartheta) 
= {\mathcal E}_{F_\gamma}(\vartheta)
\end{aligned}
$$
which proves that $\liminf_{T\to\infty}-\,\frac 1 T\E[\log \mathscr Z_T] \geq \sup_{\vartheta\in{\mathfrak m_\gamma}} {\mathcal E}_{F_\gamma}(\vartheta)$.

Now Lemma \ref{LLN} implies that, for any $s>0$, $\mathscr W(\nu_T,\Pi_s\nu_T)\to 0$ almost surely. Combining this convergence with the fact that ${\mathfrak m_\gamma}$ is compact,
 we have the (almost sure) law of large numbers 
$\mathscr W(\nu_T,{\mathfrak m_\gamma})\to 0$ almost surely w.r.t. $\P$. 
Also note that by Lemma \ref{Rewrite of the partition function}, $\frac{1}{T}\log\mathscr Z_T=\frac{M_T^\ssup{1}}{T}-\frac 1 T  \int_0^TF_\gamma(\widetilde{\bQ}_t) \d t$, where $M_T^\ssup{1}$ is a martingale with mean zero and quadratic variation given by
\begin{align*}
\mathrm{d}{\big\langle M_T^\ssup{1}\big\rangle}=\gamma^2\int_0^T\int_{\R^d}\left(\bEt\left[\kappa(y-\omega_t)\right]\right)^2\mathrm{d}y\mathrm{d}t
&\leq \gamma^2\int_0^T\int_{\R^d}\bEt\left[\kappa(y-\omega_t)^2\right]\mathrm{d}y\mathrm{d}t
=\gamma^2 T V(0),
\end{align*}
showing that $\limsup_{T\to\infty}-\,\frac{1}{T}\log\mathscr Z_T=\limsup_{T\to\infty}\frac 1 T  \int_0^TF_\gamma(\widetilde{\bQ}_t) \d t$ almost surely. Furthermore, by definition of the occupation measures $\nu_T=\frac{1}{T}\int_0^T\delta_{\widetilde{\bQ}_t}\d t$ it holds that $\frac 1 T  \int_0^TF_\gamma(\widetilde{\bQ}_t) \d t={\mathcal E}_{F_\gamma}(\nu_T)$. Since 
${\mathcal E}_{F_\gamma}(\cdot)$ is continuous on $\Mcal_1(\X)$, using the almost sure law of large numbers $\mathscr W(\nu_T,{\mathfrak m_\gamma})\to 0$, it follows that 
\begin{equation}\label{assertion-4-thm4.9.}
\limsup_{T\to\infty}-\,\frac 1 T \log \mathscr Z_T = \limsup_{T\to\infty} {\mathcal E}_{F_\gamma}(\nu_T)\leq \sup_{\vartheta\in {\mathfrak m_\gamma}} {\mathcal E}_{F_\gamma}(\vartheta) \qquad\mbox{a.s.}
\end{equation}
On the other hand, again by Lemma \ref{Rewrite of the partition function}, and as $M_T^\ssup{1}$ has mean zero, it follows that $-\, \frac 1 T\E[\log\mathscr Z_T]= \E[{\mathcal E}_{F_\gamma}(\nu_T)]$. By Lemma \ref{lemma-Phi}, both $F_\gamma$ and ${\mathcal E}_{F_\gamma}$, are non-negative and bounded from above by $\gamma^2V(0)/2$. Therefore, by reverse Fatou's lemma and \eqref{assertion-4-thm4.9.}, 
$$
\limsup_{T\to\infty}-\,\frac 1 T\E[\log\mathscr Z_T] = \limsup_{T\to\infty} \E[{\mathcal E}_{F_\gamma}(\nu_T)] \leq \E\big[\limsup_{T\to\infty} {\mathcal E}_{F_\gamma}(\nu_T)] \leq  \sup_{\vartheta\in{\mathfrak m_\gamma}} {\mathcal E}_{F_\gamma}(\vartheta)
$$
and therefore 
$\lim_{T\to\infty}-\,\frac 1 T\E[\log\mathscr Z_T]   =\sup_{\vartheta\in{\mathfrak m_\gamma}} {\mathcal E}_{F_\gamma}(\vartheta)$. 
Finally, by combining Lemma \ref{lemma-CC} and the previous arguments we have the almost sure statement 
\begin{equation}\label{varfor}
\lim_{T\to\infty}\,\frac 1 T\log\mathscr Z_{\gamma,T}=\lim_{T\to\infty}\,\frac 1 T\E[\log\mathscr Z_{\gamma,T}]
=-\sup_{\vartheta\in{\mathfrak m_\gamma}} {\mathcal E}_{F_\gamma}(\vartheta), 
\end{equation}

\smallskip

\noindent{\bf Step 4:} In this step we will conclude the proof of Theorem \ref{Theorem 2 new} using the two results stated  below for which we will 
use the notation 
\begin{equation}\label{gamma_1}
\lambda(\gamma)=\frac{\gamma^2}{2}V(0)-\sup_{\vartheta\in{\mathfrak m_\gamma}} {\mathcal E}_{F_\gamma}(\vartheta), \quad \gamma_1=\inf\big\{\gamma>0\colon \sup_{ {\mathfrak m_\gamma}} {\mathcal E}_{F_\gamma}(\cdot)>0\big\}\in[0,\infty].
\end{equation} 
\begin{prop}\label{prop3}
	Fix $d\in\N$ and $\gamma>0$.	Then the following hold:
	\begin{itemize}
		\item[(i)] It holds $\lambda(\gamma)=\frac{\gamma^2}2 V(0)$ if $\gamma\leq\gamma_1$ and $\lambda(\gamma)<\frac{\gamma^2}2 V(0)$ if $\gamma >\gamma_1$. 	
		
		\item[(ii)] Let $f_T(\gamma)=f_T(\gamma,{\dot B})=\frac{1}{T}\log Z_{\gamma,T}$ (so that $\lim_{T\to\infty}f_T(\gamma)=\lambda(\gamma)$ almost surely)		
		and 	assume that $\lambda(\gamma)$ is differentiable at $\gamma$, then $0\leq \lambda^\prime(\gamma)\leq \gamma V(0)$ and if $\gamma_T= \gamma+ o(T)$ as $T\to\infty$, then it holds that 
		\begin{equation}\label{1st}
		\lim_{T\rightarrow \infty}f^\prime_T(\gamma_T)=\lambda^\prime(\gamma)\quad \text{a.s. and in }L^1(\P)
		\end{equation}
		and
		\begin{equation}\label{Lemma 3.11}
		\lim_{T\rightarrow \infty}\widehat\bE_{\gamma,T}\Big[\big|T^{-1}\mathscr H_T(\omega)-\lambda^\prime(\gamma)\big|\Big]=0\quad\text{a.s. and in }L^1(\P).
		\end{equation}
		\item[(iii)] Now fix $d\geq 3$. Then $\gamma_1=\gamma_1(d)>0$ 
		and for $\gamma\in[0,\gamma_1]$, $\lambda(\gamma)=\frac{\gamma^2}2V(0)$. In particular, 
		$\lambda(\gamma)$ is differentiable in $[0,\gamma_1)$ and almost surely
		$$
		\lim_{T\rightarrow \infty}T^{-1}\widehat{\bE}_{\gamma,T}[\mathscr H_T(\omega)]=\lambda^\prime(\gamma)=\gamma V(0).
		$$
	\end{itemize}
\end{prop} 

We will prove Proposition \ref{prop3} after completing the proof of Theorem \ref{Theorem 2 new} below, for which we will also need 
\begin{prop}\label{prop-m}
Fix $d\in \N$ and $\gamma>0$. Then there exists a non-empty, compact subset ${\mathfrak m_\gamma}\subset \Mcal_1(\X)$ such that the supremum 
\begin{equation}\label{sup}
\sup_{\vartheta\in{\mathfrak m_\gamma}}{\mathcal E}_{F_\gamma}(\vartheta) = \sup_{\vartheta\in{\mathfrak m_\gamma}} \,\,\int_{\X} F_\gamma(\xi)\vartheta(\d\xi)
\end{equation}
 is attained, and we always have $\sup_{{\mathfrak m_\gamma}}{\mathcal E}_{F_\gamma}(\cdot) \in [0,\gamma^2(\kappa\star\kappa)(0)/2]$. Moreover, there exists $\gamma_1=\gamma_1(d)$ such that $\gamma_1>0$ if $d\geq 3$; and if  $\gamma \in (0,\gamma_1]$, then ${\mathfrak m_\gamma}=\{\delta_{\widetilde 0}\}$ is a singleton consisting of the Dirac measure at $\widetilde 0 \in \X$, and consequently in this regime $\sup_{{\mathfrak m_\gamma}}{\mathcal E}_{F_\gamma}(\cdot) =0$. If $\gamma>\gamma_1$, then 
$\sup_{{\mathfrak m_\gamma}}{\mathcal E}_{F_\gamma}>0$. 
Finally, if $\vartheta \in {\mathfrak m_\gamma}$ is a maximizer of ${\mathcal E}_{F_\gamma}(\cdot)$ and $\vartheta(\xi)>0$ for $\xi=(\widetilde\alpha_i)_{i\in I}\in \X$,  
then $\sum_{i\in I}\alpha_i(\R^d)=1$ (i.e., any maximizer  of \eqref{sup} assigns positive mass only to those 
elements of $\X$ whose total mass add up to one).
\end{prop} 
\begin{proof}
The set ${\mathfrak m_\gamma}\subset \Mcal_1(\X)$ has been defined in Lemma \ref{lemma-m} which also implies that ${\mathfrak m_\gamma}$ is non-empty and compact. The fact that the supremum $\sup_{{\mathfrak m_\gamma}}{\mathcal E}_{F_\gamma}(\cdot)$ is attained over ${\mathfrak m_\gamma}$ is a consequence of 
the continuity of ${\mathcal E}_{F_\gamma}$ (recall Lemma \ref{main step for variational}) on the closed subspace ${\mathfrak m_\gamma}$ of $\Mcal_1(\X)$
 which is compact (because $\X$ is compact). Now by Lemma \ref{lemma-Phi}, $\sup_{{\mathfrak m_\gamma}}{\mathcal E}_{F_\gamma}\leq \gamma^2 V(0)/2$
and by definition, $\sup_{{\mathfrak m_\gamma}}{\mathcal E}_{F_\gamma} \geq 0$ and equality holds by Part (i) of Proposition \ref{prop3}  if and only if $\gamma\leq \gamma_1$.
Moreover, by Part (iii) of Proposition \ref{prop3}, $\gamma_1>0$ if $d\geq 3$.

It remains to prove that, for $\gamma>\gamma_1$, a maximizer of ${\mathfrak m_\gamma}$ gives positive probability only to those elements $\xi\in\X$ which have
 total mass $1$. We again write $\sigma(\xi)=\sum_{i\in I}\alpha_i(\R^d)$ 
and using the strict concavity of $x\mapsto \frac{x}{x+(1-\sigma(\xi))}$ we get a strict inequality 
$
\E[\sigma(\xi^{\ssup{t}})]=\E\big[\frac{\mathscr G_t(\xi)}{\mathscr G_t(\xi)+(1-\sigma(\xi))}\big]<\frac{\E[\mathscr G_t(\xi)]}{\E[\mathscr G_t(\xi)]+(1-\sigma(\xi))}=\sigma(\xi),
$
where $\xi^{\ssup t}$ is defined in \eqref{alpha after time t}. 
It follows that, for any $t>0$, $
\int \sigma(\xi^\prime) \, \Pi_t(\vartheta, \d\xi^\prime)= \int \vartheta(\d\xi) \, \E[\sigma(\xi^{\ssup t})]  < \int \vartheta(\d\xi) \, \sigma(\xi)
$ if $\vartheta$ assigns positive mass to $\xi$. But on the other hand we also have $\Pi_t\vartheta=\vartheta$, since $\vartheta\in{\mathfrak m_\gamma}$. The resulting contradiction shows that $\vartheta$ assigns positive mass only to those $\xi\in\X$ that have total mass 0 or 1.

Next, we note that if ${\mathfrak m_\gamma}=\{\delta_{\tilde 0}\}$, then $\sup_{{\mathfrak m_\gamma}}{\mathcal E}_{F_\gamma}=0$ and by Proposition \ref{prop3} we have $\gamma\leq\gamma_1$. In other words, if $\gamma>\gamma_1$, there exists an element $\delta_{\tilde 0}\neq\vartheta_0\in{\mathfrak m_\gamma}$ which, by our previous remark, satisfies
$\vartheta_0(\mathscr S)=1 \text{ where } \mathscr S=\{\xi\in \X\colon \sigma(\xi)\in\{0,1\}\}$.
We also set $\mathscr S_1=\{\xi\in \X\colon \sigma(\xi)=1\}$. It now suffices to show that $\vartheta_0(\mathscr S_1)\in(0,1)$ implies that $\vartheta_0$ is not a maximizer of ${\mathfrak m_\gamma}$. Therefore, with $\vartheta_0(\cdot|\mathscr S_1)$ denoting conditional probability on $\X$, 
we claim that, whenever $\vartheta_0(\mathscr S_1)\in(0,1)$, then ${\mathcal E}_{F_\gamma}(\vartheta_0(\cdot|\mathscr S_1))>{\mathcal E}_{F_\gamma}(\vartheta_0)$. Indeed, by continuity of $F_\gamma$ we have $F_\gamma(\xi)>0=F_\gamma(\tilde 0)$ for any $\xi\neq \tilde 0$. Combining this with the fact that $\vartheta_0(\mathscr S)=1$ shows that, for those $\vartheta_0$ with $\vartheta_0(\mathscr S_1)\in(0,1)$,
\begin{align*}
{\mathcal E}_{F_\gamma}(\vartheta_0(\cdot|\mathscr S_1))=\int_{\mathscr S_1}F_\gamma(\xi)\vartheta_0(\d \xi)+\frac{1-\vartheta_0(\mathscr S_1)}{\vartheta_0(\mathscr S_1)}\int_{\mathscr S_1}F_\gamma(\xi)\vartheta_0(\d \xi)
>\int_{\mathscr S_1}F_\gamma(\xi)\vartheta_0(\d \xi)={\mathcal E}_{F_\gamma}(\vartheta_0).
\end{align*}
Since $\vartheta_0\in{\mathfrak m_\gamma} \Rightarrow \vartheta_0(\cdot|\mathscr S_1)\in{\mathfrak m_\gamma}$ $^{\diamond\diamond}$\footnote{$^{\diamond\diamond}$To see this, 
note that $\xi^\ssup{t}\in\mathscr S_1$ if and only if $\xi\in\mathscr S_1$. Thus, for any $A\subset\X$, we have $\pi_t(\xi,A)=\pi_t(\xi,A\cap \mathscr S_1)$ if $\xi\in\mathscr S_1$ and also $\pi_t(\xi,A\cap \mathscr S_1)=0$ if $\xi\notin\mathscr S_1$. The required implication now follows from these two identities, since $\Pi_t(\vartheta_0(\cdot|\mathscr S_1),A)=\frac{1}{\vartheta_0(\mathscr S_1)}\int_{\mathscr S_1}\pi_t(\xi,A)\vartheta_0(\d \xi)$ and 
\begin{align*}
\int_{\mathscr S_1}\pi_t(\xi,A)\vartheta_0(\d \xi)=\int_{\mathscr S_1}\pi_t(\xi,A\cap\mathscr S_1)\vartheta_0(\d \xi)+\int_{\mathscr S_1^C}\pi_t(\xi,A\cap\mathscr S_1)\vartheta_0(\d \xi)
=\Pi_t(\vartheta_0,A\cap\mathscr S_1).
\end{align*}}
, the above display implies that 
the element $\vartheta_0$ is not a maximizer, which 
completes the proof of Proposition \ref{prop-m}. 
\end{proof}

Let us now conclude 

\noindent{\bf Proof of Theorem \ref{Theorem 2 new} (assuming Proposition \ref{prop3}):} Recall that $Z_{\gamma,T}:=\bE_0[\e^{\gamma \mathscr H_T}]$ and note that, by Cauchy-Schwarz inequality, for any $\gamma>0$ and $T>0$, 
\begin{equation}\label{esti1}
\begin{aligned}
\sup_{\varphi \in \mathscr C_T} \frac 1 T\log \hP_{\gamma,T}[\mathscr N_{r,T}(\varphi)] 
&= \sup_{\varphi \in \mathscr C_T} \frac 1 T\log \bigg[\frac 1 {Z_T}\bE_0\bigg( \e^{\gamma \mathscr H_T}\1_{\mathscr N_{r,T}(\varphi)}\bigg)\bigg]\\
&= \sup_{\varphi \in \mathscr C_T} \frac 1 T\log \bigg[\frac 1 {\bE_0(\e^{\gamma\mathscr H_T})}\bE_0\bigg( \e^{\gamma \mathscr H_T}\1_{\mathscr N_{r,T}(\varphi)}\bigg)\bigg]\\
&= -\frac 1T \log\bE_0[\e^{\gamma\mathscr H_T}] + \sup_{\varphi \in \mathscr C_T} \frac 1 T\log \bE_0\bigg( \e^{\gamma \mathscr H_T}\1_{\mathscr N_{r,T}(\varphi)}\bigg) \\
&\leq  - \frac 1T \log\bE_0[\e^{\gamma\mathscr H_T}] + \frac 1 {2T} \log\bE_0[ \e^{2\gamma \mathscr H_T}]
+\sup_{\varphi\in \mathscr C_T} \frac 1{2T} \log  \bP_0[\mathscr N_{r,T}(\varphi)] 
\end{aligned}
\end{equation}
To handle the third term above, note that for any $\varphi\in \mathscr C_T$, and with $\omega,\omega^\prime$ denoting two independent Brownian paths, 
	$$
	\bP_0(\omega\in\mathscr N_{r,T}(\varphi))^2=\bE_0^\otimes\big[\1\big\{\omega\in\mathscr N_{r,T}(\varphi),\omega^\prime\in\mathscr N_{r,T}(\varphi)\big\}\big].
	$$
	Now if $\omega, \omega^\prime\in\mathscr N_{r,T}(\varphi)$, then $\|\omega-\omega^\prime\|_{\infty,T}\leq 2r$. In particular,
	$$
	\begin{aligned}
	\bE_0^\otimes\big[\1\big\{\omega\in\mathscr N_{r,T}(\varphi),&\omega^\prime\in\mathscr N_{r,T}(\varphi)\big\}\big]\leq \bE_0^\otimes\big[\1\big\{\|\omega-\omega^\prime\|_{\infty,T}\leq 2r\big\}\big]=\bP_0^\otimes\big(\|\omega-\omega^\prime\|_{\infty,T}\leq 2r\big)
	\\
	&=\bP_0\big(\sqrt 2\|\omega\|_{\infty,T}\leq 2r\big)=\bP_0(\omega\in\mathscr N_{\sqrt 2r,T}(0))
	\end{aligned}
	$$
	Combining the last two displays, we have,
	\begin{equation}\label{spectral}
	\sup_{\varphi\in\mathscr C_T}\log\bP_0\big(\omega\in\mathscr N_{r,T}(\varphi)\big)\leq \frac{1}{2}\log \bP_0\big(\omega\in\mathscr N_{\sqrt 2r,T}(0)\big).
	\end{equation} 
	However, for any $r>0$, the probability $\bP_0(\omega\in \mathscr N_{\sqrt 2r,T}(0))$ can be rewritten as $\bP_0(\tau>T)$, where $\tau$ denotes the first exit time of the standard Brownian motion from the ball $B_{\sqrt 2r}(0)$, and therefore by the spectral theorem for $-\frac{1}{2}\Delta$ with Dirichlet boundary condition on $B_{\sqrt 2r}(0)$, we have 	
	$\lim_{T\to\infty} \frac{1}{T}\log \bP_0(\omega\in\mathscr N_{\sqrt 2r,T}(0))= -\lambda_1(\sqrt 2r)$. Recall that by \eqref{varfor}, we have $\P$-a.s., $\lim_{T\to\infty} \frac 1 T\log Z_{\gamma,T}= \frac{\gamma^2 V(0)}2 - \sup_{\vartheta\in{\mathfrak m}_\gamma} \mathscr  E_{F_\gamma}(\vartheta)$. Thus, 
	 \eqref{esti1}-\eqref{spectral} yield, $\mathbf P$-a.s., 
	\begin{equation}\label{eq-final-esti}
	\begin{aligned}
	&\limsup_{T\to\infty} \sup_{\varphi \in \mathscr C_T} \frac 1 T\log \hP_{\gamma,T}[\mathscr N_{r,T}(\varphi)] \\
	&\leq - \frac 14 \lambda_1(\sqrt 2 r) - \lim_{T\to\infty} \frac 1 T \log Z_{\gamma,T} + \frac 12 \lim_{T\to\infty}\frac 1 T\log Z_{2\gamma,T} \\
	&= - \frac 14 \lambda_1(\sqrt 2 r) - \bigg[\frac{\gamma^2 V(0)}2- \sup_{\vartheta\in \mathfrak m_\gamma} {\mathcal E}_{F_\gamma}(\vartheta) \bigg] + \frac 12  \bigg[\frac{(2\gamma)^2 V(0)}2- \sup_{\vartheta\in \mathfrak m_{2\gamma}} {\mathcal E}_{F_{2\gamma}}(\vartheta) \bigg]= - \Theta,
	\end{aligned}
	\end{equation}
	as claimed in \eqref{esti} of Theorem \ref{Theorem 2 new}. By the first part (i) of Proposition \ref{prop3}, for any $d\in \N$ and $\gamma>0$, $\lambda(\gamma):= \frac{\gamma^2 V(0)}2 - \sup_{\vartheta\in \mathfrak m_\gamma} {\mathcal E}_{F_\gamma}(\vartheta) \in [0, \frac{\gamma^2 V(0)}2]$. Thus, using  that $\lambda(\gamma) \geq 0$ and $\lambda(2\gamma) \leq 2\gamma^2$, we have, for any $d\in \N$ and $\gamma>0$, 
	$\Theta \geq \frac 14 \lambda_1(\sqrt 2 r) - {\gamma^2 V(0)}$. Thus for any $d\in \N$, if $\frac 14 \lambda_1 (\sqrt 2 r) > \gamma^2  V(0)$ then $\Theta > 0$. 
	Since the map $(0,\infty)\ni r\mapsto \lambda_1(r)$ is decreasing, for any $d$ and $\gamma$, we find $r_0 > 0$ such that $\Theta>0$ if $r < r_0$. On the other hand, for given $d\in \N$ and $r>0$, we find $\gamma_c > 0$ such that $\Theta \geq \frac 14 \lambda_1(\sqrt 2 r) - {\gamma^2 V(0)} >0$ for any $\gamma < \gamma_c$. Finally, by part (iii) of Proposition \ref{prop3}, for $d\geq 3$ and $\gamma\in [0,\gamma_1/2]$, we have $\Theta= \frac 14\lambda_1(\sqrt 2r) - \frac{\gamma^2 V(0)}2$. $^{\ddagger\ddagger}$\footnote{$^{\ddagger\ddagger}$It could very well be that a strengthening of the argument above yields that (for $d\geq 3$ and $\gamma$ sufficiently small) the decay rate is simply $\Theta= \frac 14 \lambda_1(\sqrt 2 r)$ (without the penalization term $\gamma^2 V(0)/2$ being present). Instead of Cauchy-Schwarz bound above, we could have as well invoked H\"older's inequality with $1/p + 1/q=1$ and then optimize over $p>1$ (and $q>1$). However, the eigenvalue would then carry an extra term $1/q$ factor which would bring the decay rate closer to zero when $q$ gets larger.}
Together with a pointwise lower bound similar to \eqref{esti} which will be shown below in Proposition \ref{prop-esti-lower}, 
the proof of the first part of Theorem \ref{Theorem 2 new} is thus complete. Also we note that the required bound \eqref{eq-thick} in the second part follows directly from part (iii) of Proposition \ref{prop3}. \qed 	

\begin{prop}\label{prop-esti-lower}
Fix $d\in \N$, $r>0$. Then there is a constant $\rho\in (0,\infty)$ such that for any $\gamma>0$ and $\mathbf P$-a.s. 
\begin{equation}\label{esti-lower}
\liminf_{T\to\infty}\frac{1}{T}\log\hP_{\gamma,T}\big[\mathscr N_{r}(0)\big]\geq - \Big( \lambda_1(\frac r 2) + \rho + \frac {\gamma^2}2 V(0) + \sup_{\vartheta\in \mathfrak m_\gamma} {\mathcal E}_{F_\gamma}(\vartheta)\Big)\end{equation}
\end{prop} 
Together with the previous arguments, the above result will follow from  
\begin{lemma}\label{lemma-lower-P0}
	Fix any $d\in\N$ and any $r>0$. Then there is a constant $\rho\in (0,\infty)$ and a random variable $C(\varphi)$ for any $\varphi \in \mathscr C_T$ such that for $T$ sufficiently large, 
	$$
	 \bP_0\big\{\mathscr N_{r,T}(\varphi)\} \geq C(\varphi) \exp\big[-(\lambda_1\big(r/2\big)+\rho) T \big]
	$$
\end{lemma}
\begin{proof}
	First we write $H_T^1=\{f:f(0)=0,\int_0^T| f^\prime (s)|^2\d s<\infty\}$. For any $f \in H^1_T$, by the Cameron-Martin theorem we have 
	\begin{align*}
	\begin{split}
	\bP_0(\omega\in \mathscr N_{r/2}(f))
	&=\int\e^{\int_0^T{f^\prime}(s)\d \omega(s)-\frac{1}{2}\int_0^T|\dot{f}(s)|^2\d s}\1{\{\omega\in \mathscr N_{r/2}(0)\}}\d\bP_0(\omega)\\
	&=\e^{-\frac{1}{2}\int_0^T|{f^\prime}(s)|^2\d s}\bP_0(\omega\in \mathscr N_{r/2}(0))\bE_0\Big[\e^{\int_0^T\dot{f}(s)\d \omega(s)}\Big|\{\omega\in \mathscr N_{r/2}(0)\}\Big].
	\end{split}
	\end{align*}
	By Jensen's inequality to the above expectation and also invariance of the set $\omega\in \mathscr N_{r/2}(0)$ with respect to the map $\omega\mapsto -\omega$, we then have 	
	\begin{equation}\label{4.3}
	\bP_0(\omega\in \mathscr N_{r/2}(f)) > 	\e^{-\frac{1}{2}\int_0^T|{f^\prime }(s)|^2\d s} \,\, \bP_0(\omega\in \mathscr N_{r/2}(0)).
	\end{equation}	
	We will now handle both the terms  on the right-hand side above. Again by the spectral theorem, 
	$
	\lim_{T\rightarrow\infty}\frac{1}{T}\log\bP_0(\omega\in \mathscr N_{r/2}(0))=-\lambda_1 <0.
	$
	For any given $f\in H^1_T$, let us now handle the term $\e^{-\frac{1}{2}\int_0^T|{f^\prime}(s)|^2\d s}$ in \eqref{4.3}. For any $s <t$, with $s,t \in [0,T]$, let us define the functional 
	$$
	\begin{aligned}
	&B_{s,t}: \mathscr C_T \to \R_+, \qquad \\
	&B_{s,t}(\varphi)=\inf \bigg\{ \int_s^t|{f^\prime}(u)|^2\d u : f \in H^1_T, \,\, f(s)=\varphi(s),f(t)=\varphi(t),\sup_{u\in[s,t]}|\varphi(u)-f(u)|\leq r/2\bigg\}.
	\end{aligned}
	$$
	First remark that, for any $s < u < t $, we have 
	$$
	B_{s,t}\leq B_{s,u}+B_{u,t}.
	$$
	 That is, the map $t\mapsto B_{0,t}$ is sub-additive, and therefore by Kingman's subadditive ergodic theorem we have 
	\begin{align}
	\label{rho}
	\lim_{t\rightarrow\infty}t^{-1}B_{0,t}(\cdot)=\rho, \qquad\text{a.s.-} \bP_0
	\end{align}
	and the almost sure limit $\rho$ is deterministic. We need to show that $\rho$ is finite for which we first note that by Fatou's lemma, \eqref{rho} implies that $\rho \leq \bE_0(B_{0,1})$.  If we write $B:= B_{0,1}$, then for any $\varphi \in \mathscr C_T$ and for every fixed $f \in H^1_T$, by change of variables and using the linearity of the relation in the infimum defining $B(\cdot)$, 
	we have 
	$$
	B(\varphi+ f)= B(\varphi)+ \int_0^1 |{f^\prime}(u)|^2 \d u,
	$$
	 and therefore  
	$$\sqrt{B(\varphi+ f)}- \sqrt{B(\varphi)} \leq [\int_0^1 |{f^\prime}(u)|^2 \d u]^{1/2},
	$$ 
	meaning that  the map $\varphi \mapsto \sqrt{B(\varphi)}$ is Lipschitz. Hence, 
	by Borell's inequality, $\sqrt{B}-\mathrm{median}(\sqrt B)$ possesses Gaussian tails, implying in particular that $\bE_0((\sqrt{B})^2)=\bE_0 (B)<\infty$. As already remarked,  we have $\rho \leq \bE_0(B)$, so that together with the last upper bound we have finiteness of  the limit $\rho$ in \eqref{rho}.

Finally, note that for any fixed $T>0$ and $\varphi \in \mathscr C_T$, by lower-semicontinuity of the norm $H^1_T\ni f \mapsto (\int_0^T | f^\prime(u)|^2 \d u)^{1/2}$, there exists a (minimizing) function $f^{\ssup T}=f^{\ssup T}(\varphi)$ such that 
	$$f^{\ssup T}(0)=\varphi(0), \quad f^{\ssup T}(T)=\varphi(T), \quad B_{0,T}=\int_0^T|{f^\prime}^{\ssup T} (s)|^2\d s.
	$$ 
	Then by \eqref{4.3},
	$$
	\bP_0^\otimes(\omega\in \mathscr N_{r}(\varphi)|\varphi)\geq \bP_0^\otimes(\omega\in\mathscr N_{r/2}(f^{\ssup T})|\varphi)\geq \e^{-\frac{1}{2}B_{0,T}}\bP_0(\omega\in\mathscr N_{r/2}(0)) 
	$$
	and by \eqref{rho},
	$$
	\liminf_{T\rightarrow\infty}\frac{1}{T}\log \bP_0^\otimes(\omega\in\mathscr N_r(\varphi)|\varphi)\geq -\frac{\rho }{2}+\lim_{T\rightarrow\infty}\frac{1}{T}\log\bP_0(\omega\in\mathscr N_{r/2}(0)).
	$$
	Combining the finiteness of $\rho$ together with \eqref{spectral} now proves the desired lower bound.
	\end{proof}

\begin{remark}
A uniform exponential lower bound in Proposition \ref{prop-esti-lower} should follow from a uniform lower bound on the Wiener probability in Lemma \ref{lemma-lower-P0} (with a pre-factor $C$ which is independent of $\varphi$ on the right hand side of the bound there). 
\end{remark}

\begin{proof}[Proof of Proposition \ref{prop-esti-lower}:]
We choose $T>0$ sufficiently large so that Lemma \ref{lemma-lower-P0} holds. Let $A_T\subset \mathscr C_T$ be any set with $\bP_0(A_T)>0$. If we write $Z_T(A_T)=\bE_0[\e^{\gamma \mathscr H_T} \1_{A_T}]$, then the same arguments as in \eqref{varfor} imply that, $\P$-a.s., 
\begin{align*}
\liminf_{T\to\infty}\frac{1}{T} \log Z_T(A_T)=\liminf_{T\to\infty}\frac{1}{T}\E[ \log Z_T(A_T)].
\end{align*} 
We choose $A_T=\mathscr N_{r,T}(0)$ (which has strictly positive probability under $\bP_0$ by Lemma \ref{lemma-lower-P0}). Next, using Jensen's inequality, we note that
\begin{align*}
\E[ \log Z_T(A_T)]=\E\Big[\log\big(\bP_0(A_T)\,\,\bE_0\big[\e^{\gamma\mathscr H_T}\big| A_T\big]\big) \Big]
= \log \bP_0(A_T) +\E\Big[\log\bE_0\big[\e^{\gamma\mathscr H_T}\big| A_T\big] \Big]\\
\geq \log \bP_0(A_T) +\gamma\,\,\bE_0\big[\E[\mathscr H_T] |A_T\big]
=\log \bP_0(A_T),
\end{align*}
where for the last identity we used that the martingale $\mathscr H_T(\cdot)$ has mean $0$ under $\P$. Combining the last two displays, followed by Lemma \ref{lemma-lower-P0}, it holds
$$
\liminf_{T\to\infty}\frac{1}{T}\log \bE_0\Big[\1_{\mathscr N_r(\varphi)}\e^{\gamma\mathscr H_T-\frac{\gamma^2}{2}TV(0)}\Big]\geq \liminf_{T\to\infty}\frac{1}{T}\log \bP_0(A_T)-\frac{\gamma^2}{2}V(0)\geq -\lambda_1\big(r/2\big)-\rho-\frac{\gamma^2}{2}V(0)
$$
with $\rho=\rho(r)\in(0,\infty)$ defined in \eqref{rho}. Hence, $\P$-almost surely, 
 $$
\liminf_{T\to\infty}  \frac{1}{T}\log \hP_{\gamma,T}\big[\mathscr N_r(0)\big]\geq -\bigg(\lambda_1\big(r/2\big)+\rho+\frac{\gamma^2}{2}V(0)-\liminf_{T\to\infty}\frac 1T\log \mathscr Z_T\bigg),
$$
which proves the required lower bound. 
\end{proof}


We now owe the reader the

\begin{proof}[{\bf Proof of Proposition \ref{prop3}.}]

(i) Since $\sup_{\vartheta\in{\mathfrak m_\gamma}} {\mathcal E}_{F_\gamma}(\vartheta) \geq 0$ it follows that $\lambda(\gamma) \leq \frac{\gamma^2} 2 V(0)$ for all $\gamma\geq 0$. 
The fact that $\gamma \mapsto -\lim_{T\to\infty} \frac 1T\E[\log \mathscr Z_{\gamma,T}]$ is non-decreasing in $[0,\infty]$ and is continuous in $(0,\infty)$ can be shown 
following the arguments of \cite{CY06} for discrete directed polymers. 
It follows that $\lambda(\gamma)=\frac{\gamma^2}2 V(0)$ if $\gamma\leq\gamma_1$ and $\lambda(\gamma)<\frac{\gamma^2}2 V(0)$ if $\gamma >\gamma_1$. 	

(ii) Let us now assume that $\lambda(\gamma)$ is differentiable at $\gamma$, and note that $V=\kappa\star\kappa$ satisfies $0\leq V(\cdot)\leq V(0)$. We will prove the identity 
\begin{align}
\label{gamma integration}
\frac{1}{T}\E[\log Z_{\gamma,T}]=\int_0^\gamma r\bigg(V(0)-\frac{1}{T}\int_0^T\E\bigg[\widehat{\bE}_{r,T}^{\otimes}[V(\omega_t-\omega_t^\prime)]\bigg]\d t\bigg)\d r
\end{align}	
below. Assuming this, the statement $0\leq \lambda^\prime(\gamma)\leq \gamma V(0)$ follows from the fact that the integrand in \eqref{gamma integration} is bounded by 0 from below and by $r V(0)$ from above. 

We now prove \eqref{gamma integration} using tools from Malliavin calculus and Gaussian integration by parts introduced in the last section. First differentiating $\log Z_{\gamma,T}$ w.r.t. $\gamma$ yields
\begin{align*}
\frac{\partial}{\partial\gamma}\E\big[\log Z_{\gamma,T}\big]=\E\Big[\widehat{\bE}_{\gamma,T}[\mathscr H_T(\omega)]\Big]
=\bE_0\bigg[\E\bigg[\int_0^T\int_{\R^d}\kappa(x-\omega_t)\frac{\e^{\gamma\mathscr H_T(\omega)}}{Z_{\gamma,T}}{\dot B}(t,x)\d x\d t\bigg]\bigg].
\end{align*}
From Gaussian integration by parts (cf. \eqref{byparts}) we obtain
\begin{align*}
\bE_0\bigg[\E\bigg[\frac{\e^{\gamma\mathscr H_T(\omega)}}{Z_{\gamma,T}}\int_0^T\int_{\R^d}\kappa(x-\omega_t){\dot B}(t,x)\d x\d t\bigg]\bigg]
=\bE_0\bigg[\E\bigg[\int_0^T\int_{\R^d}\kappa(x-\omega_t)D_{t,x}\frac{\e^{\gamma\mathscr H_T(\omega)}}{Z_{\gamma,T}}\d x\d t\bigg]\bigg].
\end{align*}
We now compute the Malliavin derivative on the right-hand side above:
\begin{align*}
D_{t,x}\frac{\e^{\gamma\mathscr H_T(\omega)}}{Z_{\gamma,T}}=\gamma\kappa(x-\omega_t)\frac{\e^{\gamma\mathscr H_T(\omega)}}{Z_{\gamma,T}}-\gamma\frac{\e^{\gamma\mathscr H_T(\omega)}}{Z_{\gamma,T}^2}\bE_0\Big[\kappa(x-\omega_t^\prime)\e^{\gamma\mathscr H_T(\omega')}\Big]
\end{align*}
Here $\omega,\omega'$ denote two independent Brownian motions. Thus,
\begin{align*}
&\frac{\partial}{\partial\gamma}\E\big[\log Z_{\gamma,T}\big]
=\bE_0\bigg[\E\bigg[\int_0^T\int_{\R^d}\kappa(x-\omega_t)D_{t,x}\frac{\e^{\gamma\mathscr H_T(\omega)}}{Z_{\gamma,T}}\d x\d t\bigg]\bigg]\\
&=\gamma\E\bigg[\int_0^T\int_{\R^d}\bE_0\bigg[\kappa^2(x-\omega_t)\frac{\e^{\gamma\mathscr H_T(\omega)}}{Z_{\gamma,T}}\bigg]-\bE_0^\otimes\bigg[\kappa(x-\omega_t)\kappa(x-\omega_t^\prime)\frac{\e^{\gamma(\mathscr H_T(\omega)+\mathscr H_T(\omega'))}}{Z_{\gamma,T}^2}\bigg]\d x\d t\bigg]\\
&=\gamma T\bigg(V(0)-\frac{1}{T}\int_0^T\E\Big[\widehat{\bE}_{\gamma,T}^\otimes[V(\omega_t-\omega_t^\prime)]\Big]\d t\bigg)
\end{align*}
and the identity \eqref{gamma integration} follows. 

We now prove \eqref{1st} and \eqref{Lemma 3.11}. The first claim \eqref{1st} is an immediate consequence of the convexity of $\gamma\mapsto \log Z_{\gamma,T}$, which follows readily from H\"older's inequality. 
Then the second claim \eqref{Lemma 3.11} can be deduced further from \eqref{1st} using an idea from \cite{P10} as follows. Recall \eqref{fprime} and note that for $\gamma_0>0$,
\begin{align*}
\widehat{\bE}_{\gamma_0,T}\bigg[\bigg|\frac{\mathscr H_T(\omega)}{T}-f_T^\prime(\gamma_0)\bigg|\bigg]=\widehat{\bE}_{\gamma_0,T}\bigg[\bigg|\frac{\mathscr H_T(\omega)}{T}-\frac{\widehat{\bE}_{\gamma_0,T}[\mathscr H_T(\omega)]}{T}\bigg|\bigg]\leq \frac{1}{T}\widehat{\bE}_{\gamma_0,T}^\otimes\big|\mathscr H_T(\omega)-\mathscr H_T(\omega')\big|.
\end{align*}
Next we can rewrite, using the fundamental theorem of calculus, that for any $\gamma_0 < \gamma$, 
$$
\begin{aligned}
&\int_{\gamma_0}^\gamma \widehat\bE^\otimes_{r,T}\big[| \mathscr H_T(\omega)- \mathscr H_T(\omega^\prime)|\big] \d r \\
&= (\gamma- \gamma_0) \widehat\bE^\otimes_{\gamma_0,T}\big[| \mathscr H_T(\omega)- \mathscr H_T(\omega^\prime)|\big]+ \int_{\gamma_0}^\gamma \d r \int_{\gamma_0}^r \d \theta \frac{\partial}{\partial \theta} \widehat\bE^\otimes_{\theta,T}\big[| \mathscr H_T(\omega)- \mathscr H_T(\omega^\prime)|\big] 	
\end{aligned}
$$
and by Cauchy-Schwarz inequality, combined with the bound $(a+b)^2 \leq 2a^2+ 2b^2$, we have 
$$
\frac{\partial}{\partial \theta} \widehat\bE^\otimes_{\theta,T}\big[| \mathscr H_T(\omega)- \mathscr H_T(\omega^\prime)|\big] \leq 4 \mathrm{Var}_{\hP_{\theta,T}}\big[\mathscr H_T(\cdot)\big].	
$$
Combining the last two displays we have, for $\gamma>\gamma_0$,
\begin{align*}
\widehat{\bE}_{\gamma_0,T}^\otimes&\big|\mathscr H_T(\omega)-\mathscr H_T(\omega')\big|\\
&\leq \frac{2}{\gamma-\gamma_0}\int_{\gamma_0}^{\gamma}\widehat{\bE}_{\theta,T}\Big|\mathscr H_T(\omega)-\widehat{\bE}_{\theta,T}[\mathscr H_T(\omega)]\Big|\d \theta +4\int_{\gamma_0}^{\gamma}\mathrm{Var}_{\hP_{\theta,T}}\big[\mathscr H_T(\cdot)\big]\d \theta.
\end{align*}
By Jensen's inequality,
\begin{align*}
\bigg(\frac{1}{\gamma-\gamma_0}\int_{\gamma_0}^{\gamma}\widehat{\bE}_{\theta,T}\Big|\mathscr H_T(\omega)-\widehat{\bE}_{\theta,T}[\mathscr H_T(\omega)]\Big|\d \theta\bigg)^2
\leq \frac{1}{\gamma-\gamma_0}\int_{\gamma_0}^{\gamma}\mathrm{Var}_{\hP_{\theta,T}}\big[\mathscr H_T(\cdot)\big]\d \theta.
\end{align*}
Combining the last estimates, we have 
\begin{equation}\label{Panchenko2}
\begin{aligned}
\widehat{\bE}_{\gamma_0,T}\bigg|\frac{\mathscr H_T(\omega)}{T}-f_T^\prime(\gamma_0)\bigg|\leq 2\sqrt{\frac{\Psi_T(\gamma)}{T(\gamma-\gamma_0)}}+4\Psi_T(\gamma),\qquad\text{where}\\
\Psi_T(\gamma)=\frac{1}{T}\int_{\gamma_0}^{\gamma}\mathrm{Var}_{\hP_{\theta,T}}\big[\mathscr H_T(\cdot)\big]\d \theta=\int_{\gamma_0}^{\gamma}f_T^{\prime\prime}(\theta)\d \theta=f_T^\prime(\gamma)-f_T^\prime(\gamma_0)
\end{aligned}
\end{equation}
and in the second identity in the above display we used that $f_T^{\prime\prime}(\cdot)= \frac 1 T \mathrm{Var}_{\hP_{\cdot,T}}[\mathscr H_T]$. It remains to show that the upper bound in \eqref{Panchenko2} vanishes in the limit $T\to \infty$. 
Recall that we assumed differentiability of $\gamma_0\mapsto \lambda(\gamma_0)$, and together with convexity of  $f_T(\cdot)$, we have that for any given $\eps>0$,  we may choose $\gamma$ sufficiently close to $\gamma_0$ such that $\limsup_{T\rightarrow\infty}\Psi_T(\gamma)\leq \eps$ a.s. and in $L^1$. Hence,
$$
\lim_{T\rightarrow \infty}\widehat{\bE}_{\gamma_0,T}\bigg|\frac{\mathscr H_T(\omega)}{T}-f_T^\prime(\gamma_0)\bigg|=0\quad\text{a.s. and in }L^1
$$ 
and \eqref{Lemma 3.11} follows from \eqref{1st}.

(iii)  Recall (cf. \eqref{fprime}) that  we always have the identity 
$
\widehat{\bE}_{\gamma,T}[\mathscr H_T(\omega)]=\frac{\partial}{\partial \gamma}\log Z_{\gamma,T}.
$
The goal is to prove that if $d\geq 3$ then $\gamma_1=\inf\{\gamma\colon \sup_{{\mathfrak m_\gamma}}{\mathcal E}_{F_\gamma}>0\}>0$. Then, by part (i), for any $\gamma\in[0,\gamma_1]$ it holds
$
\lambda(\gamma)=\frac{\gamma^2}{2}V(0).
$
Then $\lambda(\gamma)$ is differentiable for any $\gamma\in[0,\gamma_1)$ and 
$
\lim_{T\rightarrow\infty}\frac{1}{T}\widehat{\bE}_{\gamma,T}[\mathscr H_T(\omega)]=\frac{\partial}{\partial \gamma}\big(\lim_{T\rightarrow\infty}\frac{1}{T}\log Z_{\gamma,T}\big)=\gamma V(0). 
$
It remains to show that $\gamma_1>0$ for $d\geq 3$. It is known that (\cite{MSZ16}) in dimension $d\geq 3$, there exists $\gamma_0(d)> 0$, such that $\mathscr Z_{\gamma,T}$ converges almost surely to zero if $\gamma>\gamma_0$ and to a non-degenerate, strictly positive random variable if $\gamma<\gamma_0$. Furthermore, in any dimension $d\geq 1$,
the event $\mathcal{V}:=\{\mathscr Z_{\gamma,T}\not\to_{T\to \infty} 0\}$  is a tail event for the process
$t\to {\dot B}(t,\cdot)$, and therefore 
$\P(\mathcal{V})\in \{0,1\}$. 
So if $\lim_{T\rightarrow\infty}\mathscr Z_{\gamma,T}>0$ almost surely, since for $x>0$, $-\log(x)<\infty$,
$\lim_{T\rightarrow\infty}\frac{1}{T}\E[-\log\mathscr Z_{\gamma,T}]\leq 0,$
That is, $\gamma_1\geq \gamma_0$, and thus $\gamma_1>0$.
\end{proof}

\subsection{\bf Proof of Corollary \ref{Theorem 3}.}\label{proof thm 3}
Recall that by Feynman-Kac formula, the solution to \eqref{she} is given by
\begin{equation}\label{SHE}
u_\eps(t,x)= \bE_x\bigg[\exp\bigg\{ \gamma(\eps,d) \int_0^{t} \int_{\R^d} \kappa_\eps\big(\omega_{s}- y\big) \, {\dot B}(t- s, \d y) \d s - \frac {\gamma(\eps,d)^2 t } 2  (\kappa_\eps\star\kappa_\eps)(0)\bigg\}\bigg].
\end{equation}
If 
\begin{equation}\label{tildeH}
\widetilde H_{\eps,t}(\omega)= \int_0^t \int_{\R^d} \kappa_\eps(\omega_s- y) {\dot B}(t-s, \d y) \ d s,
\end{equation}
the scaling property of the noise ${\dot B}$ implies that ${\dot B}(s,\ d y)\d s$ has the same law as that of $\eps^{\frac d2+1} {\dot B}(\eps^{-2} s, d(\eps^{-1} y)) \d(\eps^{-2} s)$, so that (using $\kappa_\eps(\cdot)=\eps^{-d} \kappa(\cdot/\eps)$ and changing variables $\eps^{-2}s\mapsto s$ and $\eps^{-1} y \mapsto y$), we obtain from \eqref{tildeH} that 
\begin{equation}\label{relation}
\widetilde H_{\eps,t}(\omega) \stackrel{\mathrm{(d)}}=\eps^{-\frac{d-2}{2}} \int_0^{t/\eps^2} \int_{\R^d}  \kappa\big(\eps^{-1} \omega_{\eps^2 s}- y\big) \, {\dot B}(t\eps^{-2} - s, \d y) \d s.
\end{equation}
Using now Brownian scaling, we have 
$$
\begin{aligned}
&\bE_{\frac x\eps}\bigg[\exp\bigg\{ \gamma\int_0^{t/\eps^2} \int_{\R^d} \kappa\big(\eps^{-1} \omega_{\eps^2 s}- y\big) \, {\dot B}(t\eps^{-2}- s, \d y) \d s - \frac {\gamma^2t} {2\eps^2} (\kappa\star\kappa)(0)\bigg\}\bigg] 
\stackrel{\mathrm{(d)}}= u_{\eps}(t,x).
\end{aligned}
$$
Now combining with the invariance of ${\dot B}$ w.r.t. time reversal, we have the following distributional identity of the processes:
\begin{equation}\label{dist-identity}
\{u_{\eps}(t,x)\}_{x\in \R^d} \stackrel{\ssup{\mathrm d}}= \bigg\{\mathscr Z_{\gamma,\frac t{\eps^2}}\Big(\frac x \eps\Big)\bigg\}_{x\in \R^d}, 
\end{equation}
with $\mathscr Z_{\gamma,T}$ defined in \eqref{MT}. The above distributional identity now allows us to use Theorem \ref{Theorem 2 new} to deduce Corollary \ref{Theorem 3}. 
Indeed, given a fixed $t>0$, we note that Theorem \ref{Theorem 2 new} implies
\begin{align*}
\limsup_{T\to\infty}\sup_{\varphi\in\mathscr C_{tT}}\frac{1}{tT}\log \hP_{\gamma,tT}\big[\mathscr N_{r,tT}(\varphi)\big]\leq -\Theta.
\end{align*}
%
Now Corollary \ref{Theorem 3} is a direct consequence of the above remarks once we set $\eps=T^{-1/2}$ and use that
\begin{align}\label{Brownian scaling}
\bP_0^\otimes\bigg\{\sup_{0\leq s\leq t}|\omega_s-\varphi_s|\leq r\eps\bigg\}=\bP_0^\otimes\bigg\{\sup_{0\leq s\leq tT}|\omega_s-\varphi_s|\leq r\bigg\}.
\end{align}

\subsection{Proof of Theorem \ref{localization via overlaps}.}

We will conclude the proof of Theorem \ref{localization via overlaps} in this section. Let us first introduce some notation. First recall that $\widehat\bE_T^\otimes$ denotes expectation w.r.t. the product GMC measure ${\hP}_T^\otimes$ defined in \eqref{Mprod}. Then we set 
\begin{align}
\label{definition of the set B delta}
B_\delta=\bigg\{\frac{1}{T}\int_0^T\bET^\otimes[V(\omega_s-\omega_s^\prime)]\d s\leq\delta\bigg\}.
\end{align}
Also, with $\eta_r(s,x)=\e^{-r} \eta(s(\e^{2r}-1)^{-1},x)$ if $r>0$ and $\eta_0=0$, we set 
\begin{equation}\label{Phidef}
\Phi_T(\omega,\omega^\prime, \eta_r)= \frac 1 {\widehat Z_T(\eta_r)^2} \exp\big\{\gamma\mathscr H_T(\omega,\eta_r)+ \gamma \mathscr H_T(\omega^\prime,\eta_r)\big\}
\end{equation} 
and $\widehat Z_T(\eta_r)= \bET[\e^{\gamma \mathscr H_T(\omega,\eta_r)}]$ is the normalization constant so that 
$$
\bET^\otimes\big[\Phi_T(\cdot,\cdot,\eta_r)\big]=1. 
$$

One important step for the proof of Theorem \ref{localization via overlaps} is determined by the following result. Let 
\begin{equation}\label{Idef}
I_{T,t}=\frac 1{tT} \int_0^t \d r \int_0^T \d s \bET^\otimes\big[V(\omega_s-\omega^\prime_s) \Phi_T(\omega,\omega^\prime,\eta_r)\big]
\end{equation}
so that the process $(I_{T,t})_{t\geq 0}$ is adapted to the filtration $(\mathcal H_t)_{t\geq 0}$ with $\mathcal H_t$ being the $\sigma$-algebra generated by ${\dot B}$ and $\eta$ up to time $t$.
\begin{prop}
	\label{Proposition 4.2}
	With the assumption imposed in Theorem \ref{localization via overlaps}, we have the following assertions. 
	\begin{itemize}
		\item[(a)] For any $t,\eps >0$,
		$$\lim_{T\rightarrow\infty}\P\bigg(\bigg|I_{T,t/T}-\frac{1}{t/T}\int_0^{t/T}\frac{1}{T}\int_0^T\bET^{\ssup{\Br}^\otimes}[V(\omega_s-\omega_s^\prime)]\d s\d r\bigg|>\eps\bigg)=0.$$
		\item[(b)] For any $T,\eps_1,\eps_2>0$, there exists $\delta^\prime=\delta^\prime(\gamma,t,\eps_1,\eps_2)>0$ sufficiently small that
		$$
		\P\bigg(\Big|I_{T,t/T}-\frac{1}{T}\int_0^T\bET^\otimes[V(\omega_s-\omega_s^\prime)]\d s\Big|\geq \eps_1\bigg|B_\delta\bigg)\leq \eps_2
		$$
		for all $0<\delta\leq\delta^\prime$ and $T\geq 0$.
	\end{itemize}
\end{prop} 
The proof of Proposition \ref{Proposition 4.2} will be deferred to Appendix \ref{Appendix}. Assuming the above fact we will conclude the proof of Theorem \ref{localization via overlaps}.

Let $\eps>0$ be given. Recall that we assume that $\lambda$ is differentiable and $\lambda^\prime(\gamma)<\gamma V(0)$. 
If $\rho=\frac{\gamma V(0)- \lambda^\prime(\gamma)}{\gamma}$ 
and $t$ is large enough, such that $\frac{8}{t \gamma^2 \eps^2} \leq \eps$, then  by Corollary \ref{Proposition 4.0} and \eqref{1st},

$$
\liminf_{T\rightarrow\infty}\P\bigg(\bigg|\rho -\frac{1}{t/T}\int_0^{t/T}\frac{1}{T}\int_0^T\bET^{\ssup{\Br}^\otimes}[V(\omega_s- \omega_s^\prime)]\d s\d r\bigg|\leq \eps\bigg)\geq 1-\eps 
$$

and consequently,

\begin{align}
\label{(4.5)}
\liminf_{T\rightarrow\infty}\P\bigg(\frac{1}{t/T}\int_0^{t/T}\frac{1}{T}\int_0^T\bET^{\ssup{\Br}^\otimes}[V(\omega_s-\omega_s^\prime)]\d s\d r\geq \frac{4\rho}{5}\bigg)\geq 1-\frac{\eps}{2}.	
\end{align}
Let $(I_{T,t})_t$ be the process of Proposition \ref{Proposition 4.2}. We define the sets
\begin{align*}
D&=\bigg\{\frac{1}{t/T}\int_0^{t/T}\frac{1}{T}\int_0^T\bET^{\ssup{\Br}^\otimes}[V(\omega_s-\omega_s^\prime)]\d s\d r\geq \frac{4\rho}{5}\bigg\}\\
E&=\bigg\{\frac{1}{t/T}\int_0^{t/T}\frac{1}{T}\int_0^T\bET^{\ssup{\Br}^\otimes}[V(\omega_s-\omega_s^\prime)]\d s\d r\leq \frac{3\rho}{5}\bigg\}\\
E_1&=\bigg\{\bigg|I_{T,t/T}-\frac{1}{t/T}\int_0^{t/T}\frac{1}{T}\int_0^T\bET^{\ssup{\Br}^\otimes}[V(\omega_s-\omega_s^\prime)]\d s\d r\bigg|\leq \frac{\rho}{5}\bigg\}\\
E_2&=\bigg\{\bigg|I_{T,t/T}-\frac{1}{T}\int_0^T\bET^\otimes[V(\omega_s-\omega_s^\prime)]\d s\bigg|\leq \frac{\rho}{5}\bigg\}.
\end{align*}
By Proposition \ref{Proposition 4.2} (a), $\lim_{T\rightarrow\infty}\P(E_1)=1$ and by part (b), we find $0<\delta\leq \rho/5$ small enough such that $\P(E_2|B_\delta)\geq 1/2$ for all $T\geq 0$.	Since $B_\delta\cap E_1\cap E_2\subset E$ and since $D$ and $E$ are disjoint,
$$
\P(B_\delta\cap H_1\cap H_2)\leq 1-\P(D).
$$
Further,
$$
\P(B_\delta\cap E_1\cap E_2)\geq \P(E_1)+\P(E_2\cap B_\delta)-1\geq\P(E_1)-1+\frac{\P(B_\delta)}{2}.
$$
Both inequalities together then yield 
$$
\P(B_\delta)\leq 2(2-\P(D)-\P(E_1)).
$$
Therefore, from \eqref{(4.5)} and $\lim_{T\rightarrow\infty}\P(E_1)=1$ it follows that $\limsup_{T\rightarrow\infty}\P(B_\delta)\leq\eps$ and we can deduce that, 
for every $\eps>0$ there exists $\delta>0$ sufficiently small such that
\begin{equation}\label{BC Theorem 1.4}
\limsup_{T\rightarrow\infty}\P\bigg(\frac{1}{T}\int_0^T\bET^{\otimes}[V(\omega_s-\omega_s')]\d s\leq \delta\bigg)\leq \eps.
\end{equation} 

Recall that we need to show that for every $\eps> 0$ there exist $\delta,T_0>0$ and an integer $k\in\N$ and $\omega^{\ssup{1}},...,\omega^{\ssup{k}}\in\mathscr C_\infty$ such that
	\begin{align}
	\label{BC Theorem 1.6}
	\P\bigg[\hP_{\gamma,T}^\otimes\bigg(\bigcup_{i=1}^k  \mathrm{Cov}_T(\omega^{\ssup i}, \omega^{\ssup{k+1}}) \geq \delta \bigg)\geq 1-\eps\bigg]\geq 1-\eps 
	\end{align}
	for all $T\geq T_0$, where 
\begin{align*}
&\mathrm{Cov}_T(\omega,\omega')=\frac{1}{T}\int_0^T(\kappa\star \kappa)(\omega_s-\omega^\prime_s)\d s\quad\text{and}\\
&\mathrm{Cov}_T(\omega^\prime)=\widehat\bE_T(\mathrm{Cov}_T(\omega,\omega^\prime))= \frac{1}{T}\int_0^T\int_{\R^d}\kappa(y-\omega^\prime_s)\widehat{\bE}_T[\kappa(y-\omega_s)]\d y\d s.
\end{align*}

Note that \eqref{BC Theorem 1.4} implies that	
\begin{equation}\label{claim}
\begin{aligned}
&\limsup_{T\rightarrow \infty}\E\big[\widehat{\bE}_T[\1_{\mathscr A_{T,\delta}}]\big]\leq \eps, \qquad \mbox{where }\\
& \mathscr A_{T,\delta}=\bigg\{\omega\in\mathscr C_\infty:\int_0^T\int_{\R^d}\kappa(y-\omega_s)\widehat{\bE}_T[\kappa(y-\omega_s^\prime)]\d y\d s\leq\delta T\bigg\}.
\end{aligned}
\end{equation}
	
	Indeed, fix $k\in\N$, $\delta>0$ and set $\gamma_T=\gamma\sqrt{1+\frac{k}{T}}$. For $\eta_1,...,\eta_k$ being independent copies of ${\dot B}$, we define
	$$
	\hP_{T,k}(\d \omega)=\frac{1}{Z_{T,k}}\e^{\gamma\big(\mathscr H_T(\omega,{\dot B})+\frac 1 {\sqrt T}\sum_{i=1}^k\mathscr H_T(\omega,\eta_i)\big)} \bP_0(\d \omega)
	$$
	where as usual, $Z_{T,k}$ is the normalizing constant and  $\widehat{\bE}_{T,k}$ stands for expectation with respect to the probability measure $\hP_{T,k}$.
	
	We also define
	\begin{align*}
	\mathscr A_{\delta,k}&=\bigg\{\omega\in\mathscr C_\infty:\int_0^T\int_{\R^d}\kappa(y-\omega_s)\widehat{\bE}_{T,k}[\kappa(y-\omega_s^\prime)]\d y\d s\leq \delta T\bigg\}\\
	\widetilde {\mathscr A}_{\delta,k}&=\bigg\{\omega\in\mathscr C_\infty:\int_0^T\int_{\R^d}\kappa(y-\omega_s)\widehat{\bE}_{\gamma_T,T}[\kappa(y-\omega_s^\prime)]\d y\d s\leq \delta T\bigg\}\\
	\mathscr B_{\delta,k}&=\bigg\{\int_0^T\int_{\R^d}\bigg(\widehat{\bE}_{T,k}[\kappa(y-\omega_s)]\bigg)^2\leq\delta T\bigg\}.
	\end{align*}
	
	Let $\eps>0$. There exist $K=K(\gamma,\eps)\in\N$ and $\alpha=\alpha(\gamma,\eps)>0$ and $\delta=\delta(\gamma,\eps)\in(0,1/2)$ such that, for all $T$ large enough there exists $k=k(T)\in\{0,1,...,K-1\}$ with
	$$
	\E\bigg[\widehat{\bE}_{T,k}\bigg[\1_{\mathscr A_{\delta^{4^{-k}},k}}\bigg]\1_{\mathscr B_{\alpha,k}^c}\bigg]\leq \eps/2\quad \text{and}\quad \P\big(\mathscr B_{\alpha,k}\big)\leq \eps/2.
	$$
	The latter is a consequence of \eqref{BC Theorem 1.4}. To show \eqref{claim}  (where we have $k=0$) 
	from the above estimate (where we have $k=k(T)>0$ copies of ${\dot B}$), we can estimate $\delta\leq \delta^{4^{-k}}$ and
	$$
	0\leq \gamma_T-\gamma\leq \gamma\frac{k(T)}{T}\leq \gamma\frac{K(\gamma,\eps)}{T}
	$$
	and use that 
	$
	\widehat{\bE}_{T,k}\big[\mathscr A_{\delta,k}\big]\overset{(d)}{=}\widehat{\bE}_{\gamma_T,T}\big[\widetilde{\mathscr A}_{\delta,k}\big],
	$
	which shows \eqref{claim}. 
	
	Finally, \eqref{BC Theorem 1.6} is deduced from \eqref{claim} as follows. Fix $\eps>0$. Then from \eqref{claim}, there exist $\delta>0$ small enough and $T_0>0$ large enough such that $\E\big[\widehat{\bE}_T[\1_{\mathscr A_{T,2\delta}}]\big]\leq \eps^2/2$ for any $T\geq T_0$. Then, from Markov's inequality it follows
\begin{align}
\label{6.1}
\P\big(\widehat{\bE}_T[\1_{\mathscr A_{T,2\delta}}]>\eps/2\big)\leq \eps.
\end{align}

By Paley-Zygmund inequality$^{\star\star\star}$\footnote{$^{\star\star\star}$For any random variable $X\geq 0$, $\bP[\frac X{\bE(X)} \geq \frac 12] \geq \frac 14 \frac{\bE(X)^2}{\bE(X^2)}$.}, for any $i=1,...,k$, and on the event $\{\mathrm{Cov}_T(\omega^\ssup{k+1})>2 \delta \big\}$,
\begin{align*}
\widehat{\bE}_{T}^\otimes\big[\1{\big\{\mathrm{Cov}_T(\omega^{\ssup{i}},\omega^{\ssup{k+1}})\geq \delta \big\}}\big|\omega^{\ssup {k+1}}\big]
&\geq \widehat{\bE}_{T}^\otimes\big[\1{\big\{\mathrm{Cov}_T(\omega^{\ssup{i}},\omega^{\ssup{k+1}})\geq 1/2 \mathrm{Cov}_T(\omega^\ssup{k+1}) \big\}}\big|\omega^{\ssup {k+1}}\big]\\
&\geq \frac{1}{4}\frac{\mathrm{Cov}_T(\omega^\ssup{k+1})^2}{\widehat{\bE}_T^\otimes\big[\mathrm{Cov}_T(\omega^{\ssup{i}},\omega^{\ssup{k+1}})^2\big|\omega^{\ssup{k+1}}\big]}
\geq \frac{\delta^2}{V(0)}.
\end{align*}
Hence,
\begin{align*}
\widehat{\bE}_{T}^\otimes&\bigg[\1\bigg\{\bigcap_{i=1}^k\big\{\mathrm{Cov}_T(\omega^{\ssup{i}},\omega^{\ssup{k+1}})< \delta \big\}\bigg\}\bigg|\omega^{\ssup {k+1}}\bigg]\1{\big\{\mathrm{Cov}_T(\omega^\ssup{k+1})>2 \delta  \big\}}
\leq \bigg(1-\frac{\delta^2}{V(0)}\bigg)^k\leq \e^{-\frac{\delta^2}{V(0)}k}.
\end{align*}
If we choose $k=\lceil -\delta^{-2}V(0)\log(\eps/2)\rceil\vee 0$, then 
\begin{align*}
\widehat{\bE}_{T}^\otimes&\bigg[\1\bigg\{\bigcap_{i=1}^k\big\{\mathrm{Cov}_T(\omega^{\ssup{i}},\omega^{\ssup{k+1}})< \delta \big\}\bigg\}\bigg]
\leq \frac{\eps}{2}+\widehat{\bE}_{T}^\otimes\big[\1{\big\{\mathrm{Cov}_T(\omega^\ssup{k+1})\leq 2 \delta  \big\}}\big]
=\frac{\eps}{2}+\widehat{\bE}_T[\1_{\mathscr A_{T,2\delta}}]
\end{align*}
and furthermore, 
\begin{align*}
\P\bigg(\widehat{\bE}_{T}^\otimes&\bigg[\1\bigg\{\bigcup_{i=1}^k\big\{\mathrm{Cov}_T(\omega^{\ssup{i}},\omega^{\ssup{k+1}})\geq \delta \big\}\bigg\}\bigg]\geq 1-\eps\bigg)\\
&=\P\bigg(\widehat{\bE}_{T}^\otimes\bigg[\1\bigg\{\bigcap_{i=1}^k\big\{\mathrm{Cov}_T(\omega^{\ssup{i}},\omega^{\ssup{k+1}})< \delta \big\}\bigg\}\bigg]\leq \eps\bigg)
\geq \P\Big(\widehat{\bE}_T[\1_{\mathscr A_{T,2\delta}}]\leq \eps/2\Big)\geq 1-\eps
\end{align*}
where the last inequality is due to \eqref{6.1}. This completes the proof of \eqref{BC Theorem 1.6} and therefore that of Theorem \ref{localization via overlaps}. \qed


\appendix

\section{} \label{Appendix}

\subsection{Proof of Proposition \ref{Proposition 4.2}.}

Given the input from previous sections, the arguments needed for the proof of Proposition \ref{Proposition 4.2} 
will follow the approach of \cite{BC19} adapted to our setting modulo some modifications. However, in this execution some of the arguments get simplified in our set up, thanks to the estimate 
$$
\begin{aligned}
V(x):=(\kappa \star \kappa)(x)= \int_{\R^d} \kappa (x-y) \kappa(y) \d y &\leq \bigg(\int_{\R^d} \kappa^2(x-y) \d y\bigg)^{1/2} \bigg(\int_{\R^d} \kappa^2(y) \d y\bigg)^{1/2} \\
&= \int_{\R^d} \kappa(y) \kappa(-y) \d y= (\kappa \star \kappa)(0) = V(0)
\end{aligned}
$$
for which we use that the mollifier $\kappa(\cdot)$ is a spherically symmetric function around the origin.

Recall that $\lambda(\gamma)=\lim_{T\to\infty} f_T(\gamma,\cdot)=\lim_{T\to\infty}\frac 1T\log Z_T$ and for any $\delta>0$, $\gamma>0$ and $t>0$, suitable constants $c(\gamma,t)$ and $C(\delta)$, define 

\begin{equation}\label{Mdef}
M_T^2=c(\gamma,t)\bigg(C(\delta)\bET\bigg[\Big|\lambda^\prime(\gamma)-\frac{\mathscr H_T(\omega)}{T}\Big|\bigg]+\delta Z_T^{-\frac{12t}{T}}\bigg).
\end{equation}

\begin{lemma}
	\label{Lemma 4.4.}
	Fix $t>0$ and $\eps>0$. Then the following statements hold:
	\begin{itemize}
		\item[(a)] There is $\delta=\delta(\eps)$ sufficiently small so that 
		$$
		\limsup_{T\rightarrow\infty}\E[M_T]\leq\eps. 
		$$
		\item[(b)]  For every $s\in[0,T]$ and $r\in[0,t/T]$,
		$$
		\begin{aligned}
		&\bigg\|\bET^\otimes\big[V(\omega_s-\omega^\prime_s)\Phi_T(\omega,\omega^\prime,\eta_r)-\bET^{\ssup{\Br}^\otimes}[V(\omega_s-\omega_s^\prime)]\bigg\|_{L^1(\P)}
		& \leq V(0)\E[M_T].
		\end{aligned}
		$$

		\item[(c)] There exists $\delta^\prime=\delta^\prime(\gamma,T,\eps)>0$ sufficiently small that for every $s\in[0,T]$, $r\in[0,t/T]$ and $\delta\in(0,\delta^\prime]$,
		$$
		\begin{aligned}
		&\bigg\|\bET^\otimes\big[V(\omega_s-\omega^\prime_s)\Phi_T(\omega,\omega^\prime,\eta_r)-\bET^\otimes[V(\omega_s-\omega_s^\prime)]\bigg\|_{L^1(\P,B_\delta)} 
		&\leq\eps V(0)\P(B_\delta) 
		\end{aligned}
		$$
		where $L^1(\P,B_\delta)$ is the $L^1(\P)$ norm defined on the event $B_\delta$ defined in \eqref{definition of the set B delta}. 
	\end{itemize}
\end{lemma}

We first conclude the proof of Proposition \ref{Proposition 4.2} and prove this technical fact afterwards.

\begin{proof}[{\bf Proof of Proposition \ref{Proposition 4.2} (Assuming Lemma \ref{Lemma 4.4.})}] 
	
	Let $t,\eps>0$ be fixed. By part (a)-(b) of Lemma \ref{Lemma 4.4.}, we have $\limsup_{T\rightarrow\infty}\E[M_T]\leq\eps^2$ and
	$$
	\bigg\|\bET^\otimes\big[V(\omega_s-\omega^\prime_s)\Phi_T(\omega,\omega^\prime,\eta_r)-\bET^{\ssup{\Br}^\otimes}[V(\omega_s-\omega_s^\prime)]\bigg\|_{L^1(\P)}\leq V(0)\E[M_T].
	$$
	With $I_{T,t}$ defined in \eqref{Idef} and by Lemma \ref{Lemma 4.4.}, we have 
	\begin{align*}
	\bigg\|I_{T,t/T}-\frac{1}{t/T}\int_0^{t/T}\frac{1}{T}\int_0^T\bET^{\ssup{\Br}^\otimes}[V(\omega_s-\omega_s^\prime)]\d s\d r\bigg\|_{L^1(\P)}
	\leq V(0)\E[M_T]
	\end{align*}
	and then part (a) of Proposition \ref{Proposition 4.2} follows from the Markov inequality.
	The proof of the second part is an identical application of Markov's inequality and part (c) of Lemma \ref{Lemma 4.4.} with the choice $\eps=\eps_1 \eps_2$. 
	
\end{proof}


{\bf {Proof of Part (a) and Part (b) of Lemma \ref{Lemma 4.4.}}.}
We will complete the proof in three main steps. Recall that $\eta$ is an independent copy of ${\dot B}$, while 
$$
\Br(s,\cdot)=\e^{-r} {\dot B}(s,\cdot)+ \e^{-r} \eta(s(\e^{2r}-1)^{-1}, \cdot) \text{ if }r>0,\quad \boldsymbol{{\dot B}}_0={\dot B},
$$
which has the same law as that of ${\dot B}$ and also  
$$
\eta_r(s,\cdot)= \e^{-r} \eta(s(\e^{2r}-1)^{-1},\cdot)\, \stackrel{\mathrm{(d)}}= \, \sqrt{1-\e^{-2r}} \eta(s,\cdot),\quad \eta_0=0.
$$
We will also use the simple fact that for any $c\geq 0$, if $r\leq t/T$, 
\begin{align}
\label{inequality (4.29)}
T(1-\e^{-cr})\leq Tcr\leq ct.
\end{align}

\begin{lemma}\label{lemma1}
	
	Fix $t,\gamma>0$ and $r \leq t/T$, and denote by 
	\begin{equation}\label{Yr}
	Y_r^{1/2}= \bET\bigg[\exp\big\{\gamma \big[\mathscr H_T(\omega,\Br)- \mathscr H_T(\omega,{\dot B})\big]\big\}\bigg].
	\end{equation} 
	Then there is a constant $c(\gamma,t)$ such that 
	$$
	\sup_T \E\big[Y_r^{-2}\big] \leq c(\gamma,t).
	$$
	
\end{lemma}

\begin{proof}
	
	By Jensen's inequality and from the definition of $Y_r$ we have 
	\begin{align*}
	\E_\eta[Y_r^{-2}]&\leq 
	\bET^\otimes[\e^{2\gamma^2(1-\e^{-2r})(TV(0)+\int_0^TV(\omega_s-\omega_s^\prime)\d s)}\e^{2\gamma(1-\e^{-r})(\mathscr H_T(\omega,{\dot B})+\mathscr H_T(\omega^\prime,{\dot B}))}]\\
	&\leq \e^{4\gamma^2(1-\e^{-2r})TV(0)}\bET^\otimes[\e^{2\gamma(1-\e^{-r})(\mathscr H_T(\omega,{\dot B})+\mathscr H_T(\omega^\prime,{\dot B}))}].
	\end{align*}
	by the last display, and an application of \eqref{inequality (4.29)} and the Cauchy-Schwarz inequality, it holds $\E_\eta[Y_r^{-2}]\leq c(\gamma,t)\bET[\e^{4\gamma(1-\e^{-r})\mathscr H_T(\omega,{\dot B})}]$. 
	It is also a straightforward application of Cauchy-Schwarz inequality that for any $q>0$ and $k=\lfloor \log_2\frac{T}{qt}\rfloor$ and for any $T$ large enough such that $k\geq 1$,
	$$
	\bET[\e^{q\gamma(1-\e^{-r})\mathscr H_T(\omega,{\dot B})}]\leq Z_T(\gamma)^{-\frac{1}{2^k}}(Z_T(2\gamma)^{\frac{1}{2^k}}+1).
	$$ 
	We use the inequality above for $q=4$. Then, if $T>0$ is large enough such that $k=\lfloor \log_2\frac{T}{qt}\rfloor\geq 1$ it holds 
	\begin{align*}
	\E\big[\E_\eta[Y_r^{-2}]\big]&\leq c(\gamma,t)\E\Big[Z_T(\gamma)^{-\frac{1}{2^k}}(Z_T(2\gamma)^{\frac{1}{2^k}}+1)\Big]\\
	&\leq c(\gamma,t)\Big(\e^{\frac{\gamma^2}{2^{k+1}}TV(0)}\e^{\frac{\gamma^2}{2^{k-1}}TV(0)}+\e^{\frac{\gamma^2}{2^{k+1}}TV(0)}\Big)
	\leq c(\gamma,t),
	\end{align*}
	where we used for the second inequality Cauchy-Schwarz and Jensen.
\end{proof}

\begin{lemma}\label{lemma2}
	
	Fix $t,\gamma>0$ and $r \leq t/T$, and denote by 
	
	$$
	(Y_r^\prime)^{1/2}= \bET\bigg[\exp\bigg\{\gamma \mathscr H_T(\omega,\eta_r)+ T\gamma(\e^{-r}-1)\lambda^\prime(\gamma)\bigg\}\bigg].
	$$
	Then with $Y_r$ defined in \eqref{Yr} and for any $\delta$ we find a constant $C(\delta)$ such that
	\begin{align}
	\label{inequality for Yr-Yrprime}
	\E_\eta[(Y_r-Y_r^\prime)^2]\leq c(\gamma,t)\bigg( C(\delta)\bET\bigg[\Big|\lambda^\prime(\gamma)-\frac{\mathscr H_T(\omega)}{T}\Big|\bigg]+\delta Z_T^{-\frac{12t}{T}}\bigg).
	\end{align}
	
\end{lemma}

\begin{proof}

	By Jensen's inequality and using $V(\omega_s-\omega_s^\prime)\leq V(0)$, followed by an application of \eqref{inequality (4.29)},
	\begin{align}
	\begin{split}
	\label{first inequality for Claim 4.9}
	&\E_\eta[(Y_r-Y_r^\prime)^2]\leq c(\gamma,t)\bET^\otimes\bigg[\Big(\e^{\gamma(\e^{-r}-1)(\mathscr H_T(\omega,{\dot B})+\mathscr H_T(\omega^\prime,{\dot B})-2T\lambda^\prime(\gamma))}-1\Big)^2\bigg].
	\end{split}
	\end{align}
	For any $L>0$ we find a constant such that $(\e^x-1)^2\leq C(L)|x|$ for all $x\leq L$. Also applying \eqref{inequality (4.29)} once more and using the notation $\mathscr I(r,T)=\big(\e^{-r}-1\big) (\mathscr H_T(\omega,{\dot B})+\mathscr H_T(\omega^\prime,{\dot B})-2T\lambda^\prime(\gamma))$, we have
	\begin{align}
	\begin{split}
	\label{second inequality for claim 4.9}
	\bET^\otimes\bigg[\Big(&\e^{\gamma\mathscr I(r,T)}-1\Big)^2\bigg]\\
	&\leq C(L,\gamma,t)\bET\bigg[\Big|\lambda^\prime(\gamma)-\frac{\mathscr H_T(\omega)}{T}\Big|\bigg]
	+\bET^\otimes\bigg[\Big(\e^{\gamma\mathscr I(r,T)}-1\Big)^2\1\{\gamma\mathscr I(r,T)>L\}\bigg].
	\end{split}
	\end{align}
	If $L$ is large enough such that $L\geq 4\gamma t \lambda^\prime(\gamma)$, then
	\begin{align*}
	\gamma(1-\e^{-r})(T\lambda^\prime(\gamma)-\mathscr H_T(\omega)+T\lambda^\prime(\gamma)-\mathscr H_T(\omega^\prime))>L\geq 4\gamma t \lambda^\prime(\gamma)\geq 4\gamma(1-\e^{-r})T\lambda^\prime(\gamma).
	\end{align*}
	It follows that
	\begin{align*}
	-2\gamma(1-\e^{-r})(\mathscr H_T(\omega)+\mathscr H_T(\omega^\prime))>2\gamma(1-\e^{-r})(T\lambda^\prime(\gamma)-\mathscr H_T(\omega)-\mathscr H_T(\omega^\prime))>L\geq 0
	\end{align*}
	and since the left-hand side is positive, we can use \eqref{inequality (4.29)} to get
	$
	-\frac{2\gamma t}{T}(\mathscr H_T(\omega)+\mathscr H_T(\omega^\prime))>L\geq 0.
	$
	We may now make use of the indicator. If $T$ is large enough such that $T\geq 6t$, then
	\begin{align}
	\begin{split}
	\label{third inequality for claim 4.9}
	\bET^\otimes\bigg[\Big(\e^{\gamma\mathscr I(r,T)}-1\Big)^2\1\{\gamma\mathscr G(r,T)>L\}\bigg]
	\leq \e^{-L}Z_T^{-\frac{12t}{T}}.
	\end{split}
	\end{align}
	Given $\delta>0$ we simply have to choose $L$ sufficiently large. Indeed, putting \eqref{first inequality for Claim 4.9}-\eqref{third inequality for claim 4.9} together 
	proves the claim. 
\end{proof}

Using very similar arguments as that of the proof of Lemma \ref{lemma2} and using $V(\cdot)\leq V(0)$ we can show that

\begin{lemma}\label{lemma3}
	Fix $t,\gamma>0$ and $r \leq t/T$, and denote by 
	
	\begin{equation}\label{XXprime}
	\begin{aligned}
	&X_{s,r}= \bET^\otimes\bigg[V(\omega_s-\omega_s^\prime) \exp\bigg\{\gamma\bigg(\mathscr H_T(\omega,\Br)+\mathscr H_T(\omega^\prime,\Br)-\mathscr H_T(\omega,{\dot B})- \mathscr H_T(\omega^\prime,{\dot B})\bigg)\bigg\}\bigg],\\
	&X^\prime_{s,r}= \bET^\otimes\bigg[V(\omega_s-\omega_s^\prime) \exp\bigg\{\gamma\bigg(\mathscr H_T(\omega,\eta_r)+ \mathscr H_T(\omega^\prime,\eta_r)+2T(\e^r-1)\lambda^\prime(\gamma)\bigg)\bigg\}\bigg].
	\end{aligned}
	\end{equation}
	Then for any $\delta$ we find a constant $C(\delta)$ such that
	\begin{align}
	\label{inequality for Xsr-Xsrprime}
	\E_\eta[(X_{s,r}-X_{s,r}^\prime)^2]\leq V(0)c(\gamma,t)\bigg( C(\delta)\bET\bigg[\Big|\lambda^\prime(\gamma)-\frac{\mathscr H_T(\omega)}{T}\Big|\bigg]+\delta Z_T^{-\frac{12t}{T}}\bigg).
	\end{align}
\end{lemma} 

Recall the definition of $\Phi_T$ from \eqref{Phidef}. Then 
$$
\bigg(\bET^{\ssup{\Br}^\otimes}[V(\omega_s-\omega_s^\prime)],\bET^\otimes\big[V(\omega_s-\omega^\prime_s)\Phi_T(\omega,\omega^\prime,\eta_r)\big]\bigg)=\bigg(\frac{X_{s,r}}{Y_{r}},\frac{X^\prime_{s,r}}{Y^\prime_{r}}\bigg)
$$
and as Lemma \ref{Lemma 4.4.} depends only on marginal distributions at fixed $r$, we fix $t,\eps>0$, $r\leq t/T$ and prove that
$
\big\|\frac{X_{s,r}}{Y_{r}}-\frac{X^\prime_{s,r}}{Y^\prime_{r}}\big\|_{L^1(\P)} \leq V(0)\E[M_T].
$
Note that
\begin{align*}
\E_\eta&\bigg[\bigg|\frac{X_{s,r}}{Y_r}-\frac{X_{s,r}^\prime}{Y_{r}^\prime}\bigg|\bigg]
&\leq \big(\E_\eta[Y_r^{-2}]\big)^{1/2}\big(\E_\eta[(X_{s,r}-X_{s,r}^\prime)^2]\big)^{1/2}+V(0)\big(\E_\eta[Y^{-2}_r]\big)^{1/2}\big(\E_\eta[(Y_r-Y_r^\prime)^2]\big)^{1/2}.
\end{align*}
Part (a) and Part (b) of Lemma \ref{Lemma 4.4.} follow now from Lemma \ref{lemma1}-Lemma \ref{lemma3}, if we define 
$$
M_T=c(\gamma,t)\bigg(C(\delta)\bET\bigg[\Big|\lambda^\prime(\gamma)-\frac{\mathscr H_T(\omega)}{T}\Big|\bigg]+\delta Z_T^{-\frac{12t}{T}}\bigg)^{1/2}
$$
choose $T$ large enough, such that, $12t\leq T$ and use \eqref{Lemma 3.11} which yields that $\delta$ can be chosen sufficiently small. 
\qed

{\bf {Proof of Part (c) of Lemma \ref{Lemma 4.4.}}}
Let us define

\begin{align*}
X''_{s,r}&:=\bET^\otimes\bigg[V(\omega_s-\omega_s^\prime)\e^{\gamma(\mathscr H_T(\omega,\eta_r)+\mathscr H_T(\omega^\prime,\eta_r))}\bigg]\e^{\gamma^2(\e^{-r}-1)TV(0)}\\ 
Y''_{r}&:=\bET^\otimes\bigg[\e^{\gamma(\mathscr H_T(\omega,\eta_r)+\mathscr H_T(\omega^\prime,\eta_r))}\bigg]\e^{\gamma^2(\e^{-r}-1)TV(0)}.
\end{align*}

We will show that 
$
\big\|\frac{X^{\prime\prime}_{s,r}}{Y^{\prime\prime}_{r}}-\bET^\otimes[V(\omega_s-\omega_s^\prime)]\big\|_{L^1(\P,B_\delta)} \leq \P(B_\delta) \eps V(0).
$ 
For simplicity, we write $X^{\prime\prime}=X^{\prime\prime}_{s,r}$ and $Y^{\prime\prime}=Y^{\prime\prime}_{r}$. Also using that $\frac{X''}{Y''}\leq V(0)$, it can be shown that
\begin{align}
\begin{split}
\label{first inequality for 4.4b}
\E_\eta\bigg|\frac{X''}{Y''}-\bET^\otimes[V(\omega_s-\omega_s')]\bigg|
&\leq V(0)\E_\eta|1-Y''|+\E_\eta|X''-\bET^\otimes[V(\omega_s-\omega_s')]|.
\end{split}
\end{align}
We first consider the second of the two expectations above: 
\begin{align}
\begin{split}
\label{second inequality for 4.4b}
&\E_\eta|X''-\bET^\otimes[V(\omega_s-\omega_s')]|\\
&=\E_\eta\bigg|\int_{\R^d}\bigg(\bET\bigg[\kappa(y-\omega_s)\e^{\gamma\mathscr H_T(\omega,\eta_r)-\frac{\gamma^2}{2}(1-\e^{-2r})TV(0)}\bigg]\bigg)^2-\big(\bET[\kappa(y-\omega_s)]\big)^2\d y\bigg|\\
&\leq \int_{\R^d}\E_\eta\bigg|\bigg(\bET\bigg[\kappa(y-\omega_s)\e^{\gamma\mathscr H_T(\omega,\eta_r)-\frac{\gamma^2}{2}(1-\e^{-2r})TV(0)}\bigg]\bigg)^2-\big(\bET[\kappa(y-\omega_s)]\big)^2\bigg|\d y
\end{split}
\end{align}
Using $a^2-b^2=(a+b)(a-b)$ followed by Cauchy-Schwarz inequality, we obtain 
\begin{align}
\begin{split}
\label{third inequality for 4.4b}
&\bigg(\int_{\R^d}\E_\eta\bigg|\bigg(\bET\bigg[\kappa(y-\omega_s)\e^{\gamma\mathscr H_T(\omega,\eta_r)-\frac{\gamma^2}{2}(1-\e^{-2r})TV(0)}\bigg]\bigg)^2-\big(\bET[\kappa(y-\omega_s)]\big)^2\bigg|\d y\bigg)^2\\
&\leq \int_{\R^d}\E_\eta\bigg[\bigg(\bET\bigg[\kappa(y-\omega_s)\e^{\gamma\mathscr H_T(\omega,\eta_r)-\frac{\gamma^2}{2}(1-\e^{-2r})TV(0)}\bigg]-\bET[\kappa(y-\omega_s)]\bigg)^2\bigg]\d y\\
&\times\int_{\R^d}\E_\eta\bigg[\bigg(\bET\bigg[\kappa(y-\omega_s)\e^{\gamma\mathscr H_T(\omega,\eta_r)-\frac{\gamma^2}{2}(1-\e^{-2r})TV(0)}\bigg]+\bET[\kappa(y-\omega_s)]\bigg)^2\bigg]\d y.
\end{split}
\end{align}
For the first factor we note that $\E_\eta\big[\bET\big[\kappa(y-\omega_s)\e^{\gamma\mathscr H_T(\omega,\eta_r)-\frac{\gamma^2}{2}(1-\e^{-2r})TV(0)}\big]\big]=\bET[\kappa(y-\omega_s)]$, so that it is bounded above by
\begin{align}
\begin{split}
\label{fourth inequality for 4.4b}
\bET^\otimes\bigg[V(\omega_s-\omega_s')\bigg(\e^{\gamma^2(1-\e^{-2r})\int_0^TV(\omega_s-\omega_s')\d s}-1\bigg)\bigg]
&\leq V(0)\bET^\otimes\bigg[\e^{\gamma^2(1-\e^{-2r})\int_0^TV(\omega_s-\omega_s')\d s}-1\bigg].
\end{split}
\end{align}
For the above exponential we use first inequality \eqref{inequality (4.29)} and then that there exists a constant $c(\gamma,t)$ such that, $\e^{x}-1\leq c(\gamma,t)x$ for any $0\leq x\leq c^\prime(\gamma,t)$. That proves
\begin{align}
\begin{split}
\label{fifth inequality for 4.4b}
\bET^\otimes\bigg[\e^{\gamma^2(1-\e^{-2r})TV(0)\frac{1}{TV(0)}\int_0^TV(\omega_s-\omega_s')\d s}-1\bigg]
&\leq c(\gamma,t)\bET^\otimes\bigg[\frac{1}{T}\int_0^TV(\omega_s-\omega_s')\d s\bigg]
\end{split}
\end{align}
The same argumentation for the second factor yields the upper bound $V(0)c(\gamma,t)$.
Thus, putting \eqref{second inequality for 4.4b}-\eqref{fifth inequality for 4.4b} together proves
$$
\E_\eta|X''-\bET^\otimes[V(\omega_s-\omega_s')]|\leq V(0)c(\gamma,t)\bET^\otimes\bigg[\frac{1}{T}\int_0^TV(\omega_s-\omega_s')\d s\bigg].
$$
A similar argumentation also shows that
$
\E_\eta|1-Y''|\leq c(\gamma,t)\bET^\otimes\big[\frac{1}{T}\int_0^TV(\omega_s-\omega_s')\d s\big].
$
The last two assertions together with \eqref{first inequality for 4.4b} then prove
$$
\E_\eta\bigg|\frac{X''}{Y''}-\bET^\otimes[V(\omega_s-\omega_s')]\bigg|\leq V(0)c(\gamma,t)\bET^\otimes\bigg[\int_0^TV(\omega_s-\omega_s')\d s\bigg]
$$
and, thus, for given $\eps$ we may choose $\delta$ small enough such that
$
\big\|\frac{X^{\prime\prime}_{s,r}}{Y^{\prime\prime}_{r}}-\bET^\otimes[V(\omega_s-\omega_s^\prime)]\big\|_{L^1(\P,B_\delta)} \leq \P(B_\delta) \eps V(0).
$
\qed

\noindent{\bf{Acknowledgement:}}  
It is a pleasure to thank Nathanael Berestycki (Vienna), Jason Miller (Cambridge)  and Vincent Vargas (Paris) for their encouragement, inspiration and valuable insights on an earlier draft of the manuscript.  We would also like to thank Louis-Pierre Arguin (New York) for sharing \cite{AZ14,AZ15} and interesting discussions on log-correlated fields as well as 
Nikos Zygouras (Warwick) for inspiring communications. We are especially grateful to Simon Gabriel (Warwick) for pointing out earlier inaccuracies. The research of both authors is funded by the Deutsche Forschungsgemeinschaft (DFG) under Germany's Excellence Strategy EXC 2044--390685587, Mathematics M\"unster: Dynamics--Geometry--Structure.


\end{document}